\documentclass[a4paper,12pt,reqno]{amsart}
\usepackage{lmodern,amsmath,amssymb,mathtools,amsthm}
\usepackage{enumerate,cite}
\usepackage{graphics,graphicx,geometry}
\usepackage{hyperref}
\usepackage{xcolor}
\usepackage{mathrsfs}
\usepackage{wasysym}
\usepackage{cleveref}
\usepackage{floatrow,morefloats}
\usepackage{epsfig}
\usepackage{epstopdf}
\usepackage{caption,subcaption}
\usepackage{exscale, relsize}

\newcommand*{\fminus}{\genfrac{}{}{0pt}{}{}{-}}
\newcommand*{\fplus}{\genfrac{}{}{0pt}{}{}{+}}

\numberwithin{equation}{section}

\newtheorem{definition}{Definition}
\newtheorem{Problem}{Problem}

\newtheorem{theorem}{Theorem}
\newtheorem{proposition}{Proposition}
\newtheorem{corollary}{Corollary}
\newtheorem{remark}{Remark}
\newtheorem{example}{Example}

\numberwithin{equation}{section}
\numberwithin{definition}{section}
\numberwithin{lemma}{section}
\numberwithin{theorem}{section}
\numberwithin{corollary}{section}
\numberwithin{remark}{section}
\numberwithin{example}{section}

\title{Recovering orthogonality from Quasi-type Kernel Polynomials using specific spectral transformations}

\author{VIKASH KUMAR$^\dagger$}
\address{$^\dagger$Department of Mathematics\\ Indian Institute of Technology, Roorkee-247667, Uttarakhand, India}
\email{vikaskr0006@gmail.com, vkumar4@mt.iitr.ac.in}

\author{A. Swaminathan$^\ddagger$}
\address{$^\ddagger$Department of Mathematics\\ Indian Institute of Technology, Roorkee-247667, Uttarakhand, India}
\email{mathswami@gmail.com, a.swaminathan@ma.iitr.ac.in}
\allowdisplaybreaks
\bigskip
\begin{document}

	\begin{abstract}
In this work, the concept of quasi-type Kernel polynomials with respect to a moment functional is introduced. Difference equation satisfied by these polynomials along with the criterion for orthogonality conditions are discussed. The process of recovering orthogonality for the linear combination of a quasi-type kernel polynomial with another orthogonal polynomial, which is identified by involving linear spectral transformation, is provided. This process involves an expression of ratio of iterated kernel polynomials. This lead to considering the limiting case of ratio of kernel polynomials involving continued fractions. Special cases of such ratios in terms of certain continued fractions are exhibited.
	\end{abstract}
	\maketitle
\markboth{Vikash Kumar and A. Swaminathan}{QUASI-TYPE KERNEL POLYNOMIALS}

	\section{introduction}
	
	Let $\mu$ be a non-trivial positive Borel measure with support containing infinitely many points. The support of $\mu$ having only finitely many points leads to the linear dependence of monomials in $L^2(d\mu)$-known as trivial measure. Thus we deal with the  measure $\mu$ having infinitely many points in the support. The monomials $\{x^j\}_{j=0}^{\infty}$ then become linearly independent in $L^2(d\mu)$. Applying  Gram-Schmidt process on $\{x^j\}_{j=0}^{\infty}$ one obtains certain polynomials $\{\mathcal{P}_n\}_{n\geq 0}$ satisfying
	\begin{align}\label{orthogonality}
		\mathcal{L}(\mathcal{P}_n(x)\mathcal{P}_m(x))=\int \mathcal{P}_n(x)\mathcal{P}_m(x)d\mu=\delta_{nm}.
	\end{align}
It can be noted that, by considering two sequences of complex constants  $\{\lambda_{n}\}$ and $\{c_n\}$, where $\lambda_n$'s and $c_n$'s are related to moment functional $\mathcal{L}$, the following three term recurrence relation (TTRR)\cite{Chihara book}
\begin{align}\label{TTRR}x\mathcal{P}_n(x)=\mathcal{P}_{n+1}(x)+c_{n+1}\mathcal{P}_{n}(x)+ \lambda_{n+1}\mathcal{P}_{n-1}(x),
\end{align}
with $\mathcal{P}_{-1}(x)=0$, $\mathcal{P}_{0}(x)=1$,   can also be used to obtain recursively the sequence of orthogonal polynomials $\{\mathcal{P}_n\}_{n\geq 0}$.
Favard's theorem\cite{Chihara book}  guarantees that there exists a unique linear moment functional $\mathcal{L}$ such that the orthogonality condition\eqref{orthogonality} is satisfied with respect to $\mathcal{L}$. Moreover, the definiteness property of the moment functional depends on the parameters $\lambda_{n}$ and $c_n$.

	The concept of linear combination of two consecutive members of a sequence of orthogonal polynomials was first studied by  Riesz\cite{Riesz} in 1923 in his solution to the Hamburger moment problem. Later, in 1937, Fej\'{e}r\cite{Fejer} studied the linear combination of three consecutive member of sequence of orthogonal polynomials. Finally, Shohat\cite{Shohat} generalized the concept to finite linear combination of orthogonal polynomials in the study of mechanical quadrature. The concept of quasi-orthogonal polynomial on the unit circle has been studied by Alfaro and Moral\cite{Manuel Alfaro_1994}. However, the quasi-orthogonality on the unit circle with respect to linear functional defined on the space of Laurent polynomials is not so strong in comparison to the real line case.   For further study of quasi-orthogonal polynomials, we encourage the readers to see\cite{Chihara book, Chihara, Akhiezer 1965, Dickinson, Draux1990, Draux2016}. \\
	
	In \cite{Grinshpun2004}, Grinshpun studied the necessary and sufficient conditions for the orthogonality of the linear combinations of polynomials which he called a special linear combination of orthogonal polynomials with respect to a weight function. The support of this weight function lies in an interval. These type of orthogonal families of polynomials appears
	in the solution to the problem of Peebles-Korous \cite{Natanson_CFT_vol2}, approximate solution to the Cauchy problem for the ordinary differential equations \cite{Ostrovetski_1985_ Quasi OP Appr. sol ODE} and  Gelfond's problem of the polynomials \cite{Gelfond_1954}. Grinshpun also proved that the Bernstein-Szeg\"{o} orthogonal polynomials of any kind can be written as a special linear combination of the Chebyshev polynomials of the same kind. The special feature of this representation is that the coefficients are independent of $n$. Orthogonality of the linear combination of orthogonal polynomials with constant coefficients is also discussed in \cite{When do linear,Alfaro2011 CMA}.  Furthermore, the TTRR type relation satisfied by a quasi-orthogonal polynomial of order one along with the orthogonality of quasi-orthogonal polynomials is discussed in \cite{Ismail_2019_quasi-orthogonal}.The second-order differential equations for quasi-orthogonal polynomials of order one is also addressed in  \cite{Ismail_2019_quasi-orthogonal}.  \\
	
	When we deal with the measure $d\tilde{\mu}$ of the form $d\tilde{\mu}=(x-k)d\mu$, where $k$ does not belong to the support of measure $d\mu$, we obtain a sequence of orthogonal polynomials, which we call \emph{kernel polynomials}. We refer to \cite{Chihara book,Kernel polynomials_Paco_2001, Chihara_1964_Kernel, Derevyagin_Bailey_CJM by DT} and references therein for further details in this direction.
 In this article, we define the linear combination of two consecutive terms of a sequence of kernel polynomials, which we call quasi-type kernel polynomials of order one.
 The orthogonality of these quasi-type kernel polynomials does not arise naturally. Hence the objective of this manuscript is to recover orthogonality from the given quasi-type kernel polynomials. In particular, given a quasi-type kernel polynomial, the process of identifying a related orthogonal polynomials which, with the linear combination of the quasi-type kernel polynomials provide the orthogonality. This orthogonality is same as the one given by the sequence of polynomials $\{P_n\}$ that would lead to the quasi-type kernel polynomials.

 \subsection{Organization}
  In Section \ref{sec:QKOP}, we discuss the necessary and sufficient condition for the quasi-type kernel polynomial of order one. In addition, we discuss the criterion for the orthogonality of quasi-type kernel polynomials. Section \ref{sec:Recovery of orthogonal polynomials} describes the recovery of orthogonality from the quasi-type kernel polynomial of order one and specific linear spectral transformations. Further the recovery of orthogonal polynomials from the quasi-type kernel polynomial of order two and the iterated kernel polynomials is addressed. In Section \ref{sec:Ratio}, we calculate the limiting case of the ratio of kernel polynomials. As specific cases, the ratio of certain kernel polynomials, namely the Laguerre polynomial and the Jacobi polynomial, in terms of continued fractions is also exhibited.


	\section{quasi-type kernel polynomial and orthogonality}\label{sec:QKOP}
	In this section, we discuss the known results about a linear combination of orthogonal polynomials  known as quasi-orthogonal polynomial and the polynomials generated by Christoffel transformation known as kernel polynomials. Motivated by this, we define quasi-type kernel polynomial of order one. We also give an example that satisfies the condition of the quasi-type kernel polynomials. We conclude this section with the discussion of orthogonality of quasi-type kernel polynomials.
	\begin{definition}{\rm\cite{Chihara book}}
		A non-zero polynomial $p$ is called quasi-orthogonal polynomial of order one if it is of degree at most $n+1$ and
		$$\mathcal{L}(x^mp(x))=\int x^mp(x)d\mu=0 ~~\text{for}~~ m= 0, 1,2,..., n-1.$$
	\end{definition}
	
	\begin{remark}
		Note that
		\begin{enumerate}
			\rm\item $\mathcal{L}(x^m\mathcal{P}_{n+1}(x))=0 ~~\text{for}~~ m= 0, 1,2,..., n-1$, and
			\item $\mathcal{L}(x^m\mathcal{P}_{n}(x))=0 ~~\text{for}~~ m= 0, 1,2,..., n-1.$
		\end{enumerate}
			This gives that $\mathcal{P}_{n+1}(x)$ and $\mathcal{P}_{n}(x)$ are both quasi-orthogonal polynomials of order one. Thus, one can think of $p(x)$ as a linear combination of $\mathcal{P}_{n+1}(x)$ and $\mathcal{P}_{n}(x)$.
	\end{remark}
	
	\noindent Note that this linear functional ${\mathcal{L}}$ will be employed throughout this manuscript, whenever we discuss about quasi-type kernel polynomials. Now, we will state the result which justifies the above remark.

	\begin{theorem}{\rm\cite{Chihara book}}
	Let	$\mathcal{Q}_{n+1}(x)$ be a quasi-orthogonal polynomial of order one, if, and only if, there are constants $a$ and $b$, not both zero simultaneously, such that
		$$\mathcal{Q}_{n+1}(x)= a\mathcal{P}_{n+1}(x) + b\mathcal{P}_{n}(x).$$
	\end{theorem}
For given $k \in \mathbb{C}$, we can define the new linear functional $\mathcal{L}^*$ for a polynomial $p(x)$ as
 $$ \mathcal{L}^*(p(x))= \mathcal{L}((x-k)p(x)).$$
This new linear functional is called the Christoffel transformation of $\mathcal{L}$ at $k$. We can define the corresponding kernel polynomials by the formula\cite{Chihara book}
	\begin{equation}\label{kernel}
		\mathcal{P}_n^*(k;x)=(x-k)^{-1}\left[\mathcal{P}_{n+1}(x)-\frac{\mathcal{P}_{n+1}(k)}{\mathcal{P}_n(k)} \mathcal{P}_n(x)\right] ~\text{for}~ n\geq 0 ~\text{and}~ \mathcal{P}_n(k)\neq 0.
	\end{equation}
$\{\mathcal{P}_n^*(k;x)\}_{n=0}^{\infty}$ is a monic orthogonal polynomial sequence with respect to $\mathcal{L}^*$\cite{Chihara book}, and hence by Favard's theorem, it satisfies the TTRR:
\begin{align}\label{Kernel-TTRR}x\mathcal{P}_n^*(k;x)=\mathcal{P}_{n+1}^*(k;x)+c_{n+1}^*\mathcal{P}_{n}^*(k;x)+ \lambda_{n+1}^*\mathcal{P}_{n-1}^*(k;x),
\end{align}
where
\begin{align}\label{Kernel recurrence parameters}
	\lambda_{n}^*=\lambda_{n}\frac{\mathcal{P}_{n}(k)\mathcal{P}_{n-2}(k)}{\mathcal{P}_{n-1}^2(k)}, \quad c_n^*=c_{n+1}-\frac{\mathcal{P}_{n}^2(k)-\mathcal{P}_{n-1}(k)\mathcal{P}_{n+1}(k)}{\mathcal{P}_{n-1}(k)\mathcal{P}_{n}(k)}.
\end{align}
	The study of the kernel polynomials with respect to the non-trivial probability measure on the unit circle is also an active part of research. When $k=1$, the sequence of kernel polynomials satisfies the three-term recurrence relation with their recurrence coefficients related to the positive chain sequences. The kernel polynomials on the unit circle are closely related to the para-orthogonal polynomials. For more information concerning kernel polynomials (known as Christoffel-Darboux kernel) and their asymptotic behavior, we refer to \cite{Bracciali_Ranga_Andrei_2018_JAT_CD Kernal on unit circle,Costa_Ranga_OPUC and chain sequence_JAT_2013,Swiderski_Assche_Christoffel_MOP_2022,Swiderski_Trojan_Asymptotic CD Kernel via TTRR_2021_JAT} and references therein.\\
	
		The Christoffel-Darboux identity\cite[eq. 4.9]{Chihara book} is given by
	\begin{align}{\label{CD identity}}
		\lambda_1\lambda_2...\lambda_{n+1}\sum_{j=0}^{n}\frac{\mathcal{P}_j(x)\mathcal{P}_j(x')}{\lambda_1\lambda_2...\lambda_{j+1}}=\frac{\mathcal{P}_{n+1}(x)\mathcal{P}_n(x')-\mathcal{P}_{n+1}(x')\mathcal{P}_n(x)}{x-x'},
	\end{align}
	where $\{\mathcal{P}_n(x)\}_{n\geq0}$ is a orthogonal polynomial sequence with respect to $d\mu$.\\
	With the use of \eqref{CD identity}, we can express the kernel polynomials as
	\begin{align}\label{Other form of kernel polynomial}
		\mathcal{P}_n^*(k;x)=\lambda_1\lambda_2...\lambda_{n+1}(\mathcal{P}_n(k))^{-1}\mathcal{K}_n(x,k),
	\end{align}
	where \begin{align}
		\mathcal{K}_n(x,k)=\sum_{j=0}^{n}p_j(x)p_j(k),~~ p_n(x)=(\lambda_1\lambda_2...\lambda_{n+1})^{-1/2}\mathcal{P}_n(x).
	\end{align}
	For a fixed $k$, we can easily deduce from the Christoffel-Darboux identity that $(x-k)\mathcal{K}_n(x,k)$ is a quasi-orthonormal polynomial of order one. On the other hand if we replace fixed number $k$ by variable $u$ then we can discuss the orthogonality of the polynomials $\{(x-u)\mathcal{K}_n(x,u)\}_{n\geq0}$.
	
	\begin{proposition}
		Let $d\mu$ be a positive Borel measure on $\mathbb{R}$ with finite moments. Then the sequence $\{(x-u)\mathcal{K}_n(x,u)\}_{n\geq0}$ forms an orthogonal polynomial sequence with respect to the product measure $d\mu(\mathbb{R}\times\mathbb{R})$ on $L^2(\mathbb{R}^2,d\mu(\mathbb{R}\times\mathbb{R}))$.
		
		\begin{align}
			\iint|x-u|^2\mathcal{K}_n(x,u)\mathcal{K}_m(x,u)d\mu(u)d\mu(x)=\begin{cases}
				0 & \text{for}~ m\neq n\\
				2\lambda_{n+1}^2 &\text{for}~ m=n.
			\end{cases}
		\end{align}
	\end{proposition}
	\begin{proof}
		\begin{align*}
			&\iint(x-u)^2\mathcal{K}_n(x,u)\mathcal{K}_m(x,u)d\mu(u)d\mu(x)\\
			&=\lambda_{n+1}^2\iint(p_{n+1}(x)p_n(u)-p_{n+1}(u)p_n(x))(p_{m+1}(x)p_m(u)-p_{m+1}(u)p_m(x))d\mu(u)d\mu(x)\\
			&=\lambda_{n+1}^2\int\left[p_{n+1}(x)p_{m+1}(x)\int p_n(u)p_m(u)d\mu(u)-p_{m+1}(x)p_n(x)\int p_{n+1}(u)p_m(u)d\mu(u)\right.\\
			&\left.-p_{n+1}(x)p_m(x)\int p_{m+1}(u)p_n(u)d\mu(u)+p_{n}(x)p_m(x)\int p_{m+1}(u)p_{n+1}(u)d\mu(u)\right]d\mu(x).
		\end{align*}
		
		For $m\leq n-2$, we have
		\begin{align*}
			\iint(x-u)^2\mathcal{K}_n(x,u)\mathcal{K}_m(x,u)d\mu(u)d\mu(x)=0.
		\end{align*}
		
		For  $m=n-1$, using the orthonormal property of $p_n$. We have
		\begin{align*}
			\iint(x-u)^2\mathcal{K}_n(x,u)\mathcal{K}_m(x,u)d\mu(u)d\mu(x)=-\lambda_{n+1}^2\int p_{n+1}(x)p_{n-1}(x)d\mu(x)=0.
		\end{align*}
		
		For $m=n$, we have
		\begin{align*}
			&\iint(x-u)^2\mathcal{K}_n(x,u)\mathcal{K}_m(x,u)d\mu(u)d\mu(x)=\lambda_{n+1}^2\left(\int p_{n+1}^2(x)d\mu(x)+\int p_{n}^2(x)d\mu(x)\right)\\
			&\hspace{7.1cm}=2\lambda_{n+1}^2.
		\end{align*}
		This completes the proof.
	\end{proof}
	
	\noindent Now, we give the following definition:
	\begin{definition}
		Let $\{\mathcal{P}_n^*(k;x)\}_{n=0}^{\infty}$ be the sequence of kernel polynomials which exists for some $k \in \mathbb{C}$, and forms an orthogonal polynomial sequence with respect to $\mathcal{L}^* $. A non-zero polynomial $\mathcal{Q}_{n+1}^*(k;\cdot)$ is called a quasi-type kernel polynomial of order one if it is of degree at most $n+1$ and $\mathcal{L}^*(x^m \mathcal{Q}_{n+1}^*(k;x))=0$ for $m=0,1,...,n-1$.
	\end{definition}

\begin{remark}Note that
	\begin{enumerate}
	\rm	\item 	$\mathcal{L}^*(x^m\mathcal{P}_{n+1}^*(k;x))=0 ~~\text{for}~~ m= 0, 1,2,...,n-1$, and
		\item $\mathcal{L}^*(x^m\mathcal{P}_{n}^*(k;x))=0 ~~\text{for}~~ m= 0, 1,2,...,n-1$,
	\end{enumerate}
so that both $\mathcal{P}_{n+1}^*(k;x)$ and $\mathcal{P}_{n}^*(k;x)$  are quasi-type kernel polynomials of order $1$.
	\end{remark}

\begin{remark}
	In general, we can say a  polynomial $\mathcal{Q}_{n+1}^*(k;\cdot)\neq 0$ is called quasi-type kernel polynomial of order $l\geq1$ if, and only if, it is of degree at most $n+1$, $n\geq l+1$ and $\mathcal{L}^*(x^m \mathcal{Q}_{n+1}^*(k;x))=0$ for $m=0,1,...,n-l$.
\end{remark}

	\begin{theorem}\label{quasi-type kernel}
		$\mathcal{Q}_{n+1}^*(k;x)$ is a quasi-type kernel polynomial of order $1$, if, and only if, there are constants a and b, not zero simultaneously, such that
		\begin{align*}
\mathcal{Q}_{n+1}^*(k;x)= a\mathcal{P}_{n+1}^*(k;x) + b\mathcal{P}_{n}^*(k;x).
\end{align*}
	\end{theorem}

	\begin{proof}
		If $\mathcal{Q}_{n+1}^*(k;x)$ is a quasi-type kernel polynomial of order $1$, then for some constant $c_0,c_1,...c_{n+1}$, we can write
		$$\mathcal{Q}_{n+1}^*(k;x)= \sum\limits_{m=0}^{n+1}c_m \mathcal{P}_{m}^*(k;x)$$
		with $c_m= \frac{\mathcal{L}^*[\mathcal{Q}_{n+1}^*(k;x)\mathcal{P}_m^*(k;x)]}{\mathcal{L}^*[{\mathcal{P}_m^*}^2(k;x)]}$ and hence, $c_m=0$ for $m\in \{0,1,...,n-1\}$.  Thus, we get $\mathcal{Q}_{n+1}^*(k;x)= a\mathcal{P}_{n+1}^*(k;x) + b\mathcal{P}_{n}^*(k;x)$. Conversely, If $a$ and $b$ are not simultaneously zero, then
		\begin{align*}
			\mathcal{L}^*(x^m \mathcal{Q}_{n+1}^*(k;x))=& \mathcal{L}^*(ax^m\mathcal{P}_{n+1}^*(k;x))+bx^m\mathcal{P}_{n}^*(k;x)\\
			=& a\mathcal{L}^*(x^m\mathcal{P}_{n+1}^*(k;x))+b\mathcal{L}^*(x^m\mathcal{P}_{n}^*(k;x))\\
			=& 0 ~~\text{for}~~ m=0,1,...,n-1.
		\end{align*}
This completes the proof.		
	\end{proof}

	\begin{remark}
	In general,	we can extend the above theorem for order l and say $\mathcal{Q}_{n+1}^*(k;x)$ is a monic quasi-type kernel polynomial of order $l$  if
		\begin{align*}
			\mathcal{Q}_{n+1}^*(k;x)= \mathcal{P}_{n+1}^*(k;x)+\sum\limits_{m=1}^{l}\alpha_{m} \mathcal{P}_{n-m+1}^*(k;x) ~\text{for}~ n\geq l+1.
		\end{align*}
	\end{remark}

	Next, we consider an example which supports Theorem \ref{quasi-type kernel}. For this, first we easily show that  polynomial $\mathcal{P}_{n+1}(x)$ is a quasi-type kernel polynomial of order one with respect to $\mathcal{L}^* $. Indeed, $\mathcal{P}_{n+1}(x)$ can be written as a linear combination of $\mathcal{P}_{n+1}^*(k;x)$ and $\mathcal{P}_{n}^*(k;x)$ with constant coefficients in the following result using TTRR \eqref{TTRR} satisfied by orthogonal polynomial $\mathcal{P}_n(x)$. Note that the same was established in \cite[eq. 2.5]{Kernel polynomials_Paco_2001} using Christoffel-Darboux kernel \eqref{CD identity}.

\begin{proposition}\label{Prop:Quasi Kernel TTRR for CDKernel}
	Let $\{\mathcal{P}_{n}^*(k;x)\}_{n=0}^\infty$ be a sequence of monic orthogonal polynomials with respect to $\mathcal{L}^* $ which exists for some $k\in \mathbb{C}$. Then we can write $\mathcal{P}_n(x)$ in terms of linear combinations of kernel polynomials as follows:
	\begin{align}
		\label{OP in terms of Kernel}
		\mathcal{P}_{n+1}(x)= \mathcal{P}_{n+1}^*(k;x) -\dfrac{\mathcal{P}_{n}(k)}{\mathcal{P}_{n+1}(k)}\lambda_{n+2}\mathcal{P}_{n}^*(k;x),
	\end{align}
	where $\lambda_{n+2}$ is a strictly positive constant in TTRR \eqref{TTRR}.
\end{proposition}
\begin{proof}Using equation \eqref{kernel}, we can write
	\begin{align*}&\mathcal{P}_{n+1}^*(k;x)+D_{n+1}\mathcal{P}_{n}^*(k;x)\\
		&=\frac{1}{x-k}\left[\mathcal{P}_{n+2}(x)-\frac{\mathcal{P}_{n+2}(k)}{\mathcal{P}_{n+1}(k)}\mathcal{P}_{n+1}(x)+D_{n+1}\mathcal{P}_{n+1}(x)-D_{n+1}\frac{\mathcal{P}_{n+1}(k)}{\mathcal{P}_{n}(k)}\mathcal{P}_{n}(x)\right]\\
		&=\frac{1}{x-k}\left[\mathcal{P}_{n+2}(x)-\left( x-k-D_{n+1}+\frac{\mathcal{P}_{n+2}(k)}{\mathcal{P}_{n+1}(k)}\right)\mathcal{P}_{n+1}(x)-D_{n+1}\frac{\mathcal{P}_{n+1}(k)}{\mathcal{P}_{n}(k)}\mathcal{P}_{n}(x)\right]+\mathcal{P}_{n+1}(x),
	\end{align*}
	by substituting
	\begin{equation*}
		D_{n+1}=-\lambda_{n+2}\frac{\mathcal{P}_{n}(k)}{\mathcal{P}_{n+1}(k)},
	\end{equation*}
	we can write the above equation as
	\begin{align*}
		\noindent&\displaystyle\mathcal{P}_{n+1}^*(k;x)+D_{n+1}\mathcal{P}_{n}^*(k;x)=\frac{1}{x-k}\left[\mathcal{P}_{n+2}(x)-\left( x-c_{n+2}\right)\mathcal{P}_{n+1}(x)+\lambda_{n+2}\mathcal{P}_{n}(x)\right]+\mathcal{P}_{n+1}(x),
	\end{align*}
	where
	\begin{align*}
		\lambda_{n+2}=-D_{n+1}\frac{\mathcal{P}_{n+1}(k)}{\mathcal{P}_{n}(k)},~~ c_{n+2}=k+D_{n+1}-\frac{\mathcal{P}_{n+2}(k)}{\mathcal{P}_{n+1}(k)}.
	\end{align*}
	Using TTRR \eqref{TTRR}, we get the desired result.
\end{proof}


\begin{example}\label{Kernel of Chebyshev and reverse}
	Let $\{\mathcal{C}_n(x)\}_{n=0}^{\infty}$ be a sequence of polynomials which forms an orthogonal polynomial sequence with respect to the Chebyshev measure $d\mu=(1-x^2)^{-1/2}dx$ with compact support $[-1,1]$. This is referred to as \emph{Chebyshev polynomial} of first kind. The corresponding monic Chebyshev polynomial can be written as
	\begin{align*}
		\hat{\mathcal{C}}_0(x)=& \mathcal{C}_0(x),\\
		\hat{\mathcal{C}}_{n+1}(x)=&2^{-n}\mathcal{C}_{n+1}(x), ~~ n\geq 0.
	\end{align*}
 The monic polynomial $\hat{\mathcal{C}}_{n}(x)$ satisfies the  following TTRR
 \begin{equation*}
 	\begin{split}
 	\hat{\mathcal{C}}_{n+1}(x)=& x\hat{\mathcal{C}}_{n}(x)-\frac{1}{4}\hat{\mathcal{C}}_{n-1}(x), ~~n\geq 2,\\
 	\hat{\mathcal{C}}_{2}(x)=&x\hat{\mathcal{C}}_{1}(x)-\frac{1}{2}\hat{\mathcal{C}}_{0}(x)
 	\end{split}
 \end{equation*}
with initial data $\hat{\mathcal{C}}_{0}(x)=\mathcal{C}_{0}(x)=1,~~ \hat{\mathcal{C}}_{1}(x)=\mathcal{C}_{1}(x)=x.$

The kernel of the Chebyshev polynomials for $k\leq-1$ and $k\geq 1$ is given by\rm{\cite[eq. 7.5]{Chihara book}}
\begin{align*}
	\hat{\mathcal{C}}_{n+1}^*(k;x)=\frac{1}{x-k}\left[\hat{\mathcal{C}}_{n+2}(x)-\frac{\hat{\mathcal{C}}_{n+2}(k)}{\hat{\mathcal{C}}_{n+1}(k)}\hat{\mathcal{C}}_{n+1}(x)\right].
\end{align*}
$\{\hat{\mathcal{C}}_{n}^*(k;x)\}_{n=0}^{\infty}$ is a monic orthogonal polynomial sequence with respect to the quasi definite linear functional $\mathcal{L}^*.$
Then,
by \eqref{OP in terms of Kernel},
we say that 	$\hat{\mathcal{C}}_{n+1}(x)$ is a quasi-type kernel polynomial of order one. Indeed,
\begin{align*}
	\hat{\mathcal{C}}_{n+1}(x)= \hat{\mathcal{C}}_{n+1}^*(k;x)-\frac{1}{4}\frac{\hat{\mathcal{C}}_{n}(k)}{\hat{\mathcal{C}}_{n+1}(k)}\hat{\mathcal{C}}_{n}^*(k;x).
\end{align*}
In addition, from the above equation it is natural to ask about the behavior of the ratio of Chebyshev polynomial and Chebyshev kernel polynomial.
In particular, for $k=1$, we have
\begin{align*}
	\hat{\mathcal{C}}_{n+1}(x)= \hat{\mathcal{C}}_{n+1}^*(1;x)-\frac{1}{2}\hat{\mathcal{C}}_{n}^*(1;x).
\end{align*}
By using Corollary \ref{Ratio kernel Chebyshev}, we get
\begin{align*}
	\lim\limits_{x\rightarrow 1}\frac{\hat{\mathcal{C}}_{n+1}(x)}{\hat{\mathcal{C}}_{n+1}^*(1;x)}=\frac{4}{3+2n}.
\end{align*}
\end{example}
\subsection{Recurrence relation and orthogonality}
It is known that we do not have the orthogonality of quasi-orthogonal polynomials with respect to $\mathcal{L}$, although it is interesting to obtain the difference equation similar to TTRR of quasi-orthogonal polynomials. In \cite{Ismail_2019_quasi-orthogonal}, Ismail and Wang discussed the TTRR type relation for quasi-orthogonal polynomials. In the next result, we generalize their result to obtain the difference equation with variable coefficients of quasi-type kernel polynomials.

\begin{theorem}
	Let $\mathcal{Q}_{n+1}^*(k;x)$ be a monic quasi-type kernel polynomial of order $1$. Then $\mathcal{Q}_{n+1}^*(k;x)$ satisfy the difference equation
	\[\mathcal{J}_n(x)\mathcal{Q}_{n+2}^*(k;x)=\left[\mathcal{D}_{n+1}(x)\mathcal{J}_n(x)-b\mathcal{J}_{n+1}(x)\right]\mathcal{Q}_{n+1}^*(k;x)-\lambda_{n+1}^*\mathcal{J}_{n+1}(x)\mathcal{Q}_{n}^*(k;x),\]
	where
	\[\mathcal{D}_{n+1}(x)=x-c_{n+2}^*+b,~~ \mathcal{J}_{n+1}(x)=b\mathcal{D}_{n}(x)+ \lambda_{n+1}^*.\]
\end{theorem}

\begin{proof}
	By the definition of $\mathcal{Q}_{n+1}^*(k;x)$, we have
	\begin{align}\label{monic quasi kernel}
		\mathcal{Q}_{n+1}^*(k;x)=\mathcal{P}_{n+1}^*(k;x)+b\mathcal{P}_n^*(k;x).
	\end{align}
	By using \eqref{Kernel-TTRR}, we can write \eqref{monic quasi kernel} as
	\begin{align}\label{Quasi kernel using TTRR}
		\mathcal{Q}_{n}^*(k;x)=\mathcal{P}_{n}^*(k;x)+b\mathcal{P}_{n-1}^*(k;x)
		=-\frac{b}{\lambda_{n+1}^*}\mathcal{P}_{n}^*(k;x)+\left((x-c_{n+1}^*)\frac{b}{\lambda_{n+1}^*}+1\right)\mathcal{P}_{n}^*(k;x).
	\end{align}
We can write the equations \eqref{monic quasi kernel} and \eqref{Quasi kernel using TTRR}, in the matrix form as follows
	\[\begin{pmatrix}
		\mathcal{Q}_{n+1}^*(k;x)\\
		\mathcal{Q}_{n}^*(k;x)
	\end{pmatrix}=\begin{pmatrix}
	1 & b\\
	-\frac{b}{\lambda_{n+1}^*} & (x-c_{n+1}^*)\frac{b}{\lambda_{n+1}^*}+1
\end{pmatrix}\begin{pmatrix}
\mathcal{P}_{n+1}^*(k;x)\\
\mathcal{P}_{n}^*(k;x)
\end{pmatrix}.\]
Since the right side of the matrix is invertible, we have

\begin{align}\label{Matrix form Kernel in terms of Quasi-kernel}
	\begin{pmatrix}
		\mathcal{P}_{n+1}^*(k;x)\\
		\mathcal{P}_{n}^*(k;x)
	\end{pmatrix}=\frac{\lambda_{n+1}^*}{b^2+\lambda_{n+1}^*+(x-c_{n+1}^*)b}\begin{pmatrix}
		(x-c_{n+1}^*)\frac{b}{\lambda_{n+1}^*}+1 & -b\\
		\frac{b}{\lambda_{n+1}^*} & 1
	\end{pmatrix}\begin{pmatrix}
		\mathcal{Q}_{n+1}^*(k;x)\\
		\mathcal{Q}_{n}^*(k;x)
	\end{pmatrix}
\end{align}
Further, using \eqref{Kernel-TTRR}, we write
\begin{align*}
	\mathcal{Q}_{n+2}^*(k;x)
	=(x-c_{n+2}^*+b)\mathcal{P}_{n+1}^*(k;x)-\lambda_{n+2}^*\mathcal{P}_{n}^*(k;x).
\end{align*}
Again, we can use \eqref{Matrix form Kernel in terms of Quasi-kernel} to obtain the expression of  $\mathcal{Q}_{n+2}^*(k;x)$ in terms of $\mathcal{Q}_{n+1}^*(k;x)$ and $\mathcal{Q}_{n}^*(k;x)$ as

\begin{align*}
	\mathcal{Q}_{n+2}^*(k;x)&=\frac{\lambda_{n+1}^*}{b^2+\lambda_{n+1}^*+(x-c_{n+1}^*)b}\left((x-c_{n+2}^*+b)\left[\left((x-c_{n+1}^*)\frac{b}{\lambda_{n+1}^*}+1\right)\mathcal{Q}_{n+1}^*(k;x)\right.\right.\\
	&\left.\left.\hspace{4.5cm}-b\mathcal{Q}_{n}^*(k;x)\right]-\lambda_{n+2}^*\left(\frac{b}{\lambda_{n+1}^*}\mathcal{Q}_{n+1}^*(k;x)+\mathcal{Q}_{n}^*(k;x)\right)\right).
	\end{align*}
After simplifying the above equation, we obtain the desired result
\[\mathcal{J}_n(x)\mathcal{Q}_{n+2}^*(k;x)=\left[\mathcal{D}_{n+1}(x)\mathcal{J}_n(x)-b\mathcal{J}_{n+1}(x)\right]\mathcal{Q}_{n+1}^*(k;x)-\lambda_{n+1}^*\mathcal{J}_{n+1}(x)\mathcal{Q}_{n}^*(k;x),\]
where
\[\mathcal{D}_{n+1}(x)=x-c_{n+2}^*+b,~~ \mathcal{J}_{n+1}(x)=b\mathcal{D}_{n}(x)+ \lambda_{n+1}^*.\]
This completes the proof.
\end{proof}
Next, we discuss the necessary and sufficient conditions for orthogonality of quasi-type kernel polynomial of order $l$. One may prove Theorem \ref{orthogonality of QK} in the same line as \cite[Theorem 1]{When do linear}, and hence we omit the proof.
\begin{theorem}\label{orthogonality of QK}
	Suppose $\{\mathcal{P}_n(x)\}_{n=0}^{\infty}$ be a sequence of monic orthogonal polynomials with respect to a quasi definite linear functional $\mathcal{L}$ and suppose $\{\mathcal{P}_n^*(k;x)\}_{n=0}^{\infty}$ be a sequence of kernel polynomials generated by Christoffel transformation $\mathcal{L}^*$ at $k$, which satisfy TTRR \eqref{Kernel-TTRR} with recurrence parameters  $c_{n+1}^*$, $\lambda_{n+1}^*$ given by \eqref{Kernel recurrence parameters}. Further, let $\{\mathcal{Q}_n^*(k;x)\}_{n=0}^{\infty}$ be a sequence of quasi-type kernel polynomials
	\begin{align*}
		\mathcal{Q}_n^*(k;x)= \mathcal{P}_n^*(k;x)+ \sum\limits_{m=1}^{l}\alpha_{m} \mathcal{P}_{n-m}^*(k;x) ~\text{for}~ n\geq l+1,
	\end{align*}
where $\{\alpha_m\}_{m=1}^l$ are scalars with nonzero value of  $\alpha_{l}$. Then $\{\mathcal{Q}_n^*(k;x)\}_{n=0}^{\infty}$ is monic orthogonal with respect to a linear functional, if, and only if, the following conditions hold:
\begin{enumerate}
	\item[\rm{(i)}] The polynomials $\mathcal{Q}_m^*(k;x)$ satisfy a TTRR given by
	\begin{align*}
	\mathcal{Q}_{m+1}^*(k;x)+(x-\tilde{c}_{m+1}^*)\mathcal{Q}_m^*(k;x)+\tilde{\lambda}_{m+1}^*\mathcal{Q}_{m-1}^*(k;x)=0,
	\end{align*}
 with $\tilde{\lambda}_{m+1}^*\neq 0$ for $m\in \{0,1,2,...,l\}$.

	\item[\rm{(ii)}] \text{For} $n>l+1$,
	
	$\lambda_{n+1}^* -\lambda_{n-l+1}^*=\alpha_1(c_{n+1}^*-c_n^*)= \neq0$,
	
	$\alpha_m(c_{n-m+1}^*-c_{n+1}^*)+\alpha_{m-1}[\lambda_{n-m+2}^*-\lambda_{n+1}^*-(\alpha_1(c_n^*-c_{n+1}^*))]=0, \quad m\in\{1,2,...,l\}.$
	 \item[\rm{(iii)}] For $m\in \{1,...,l-1\}$,
	
	  $\lambda_{l+2}^*  \neq \alpha_1(c_{l+2}^*-c_{l+1}^*),$
	
	 $\alpha_{m+1}(c_{l-m+1}^*-c_{l+2}^*)+ \alpha_{m}\lambda_{l-j+2}^*=\alpha_m^{(l)}[\lambda_{l+2}^*-\alpha_1(c_{l+1}-c_{l+2}^*)],$
	
	 $\alpha_l^{(l)}\lambda_{l+2}^*+\alpha_1\alpha_l^{(l)}(c_{l+1}-c_{l+2}^*)=\alpha_l\lambda_2^*$,
\end{enumerate}
 where $\alpha_m^{(l)}, m\in \{1,2,,...,l\}$, represents the constant coefficients  of $\mathcal{P}^*_{l-m}(k;\cdot)$ in the Fourier representation of $\mathcal{Q}_l^*(k;\cdot)$.

Moreover, for $n\geq l+1$, we have
\begin{align*}
	\tilde{c}_{n+1}^*=c_{n+1}^*, \quad \tilde{\lambda}_{n+1}^*=\lambda_{n+1}^* +\alpha_1(c_n^*-c_{n+1}^*),
\end{align*}
where $\tilde{c}_{n+1}^*$ and $\tilde{\lambda}_{n+1}^*$ are the recurrence coefficients in the TTRR expansion of $\mathcal{Q}_n^*(k;\cdot)$.
\end{theorem}


\section{Recovery of orthogonal polynomials}\label{sec:Recovery of orthogonal polynomials}
In this section, our primary goal is to recover the orthogonality of the polynomials which are the linear combination of polynomials generated by Darboux transformations and quasi-type kernel polynomials of orders 1 and 2 via suitable coefficients. In this process, we identify the unique sequences of constants that are necessary to recover such orthogonal polynomials.
\\
\subsection{Christoffel transformation}
  The relations among the quasi-orthogonal polynomials, monic orthogonal polynomial sequence and kernel polynomials are discussed in \cite{Derevyagin_Bailey_CJM by DT}.
 In Theorem \ref{QK+OP+KP}, we recover the polynomials $\mathcal{P}_{n}(x)$ from the linear combination of polynomial generated by Christoffel transformation and quasi-type kernel polynomials of order one with rational coefficients. We identify two sequences of parameters that are responsible for obtaining  $\mathcal{P}_{n}(x)$. We work with the monic quasi-type kernel polynomials of order one for some $k\in \mathbb{C},$ which can be defined as
$\mathcal{T}_{n}^*(k,x)=\mathcal{P}_{n}^*(k;x) + B_n\mathcal{P}_{n-1}^*(k;x).$

%
%
%
%
%
%
%
\begin{theorem}\label{QK+OP+KP}
	
	Let $\{\mathcal{P}_n(x)\}_{n=0}^{\infty}$ be a monic orthogonal polynomial sequence with respect to the positive definite linear functional $\mathcal{L}$. Let $\mathcal{T}_{n}^*(k_1,x)$ be a monic quasi-type kernel polynomial of order one for some $k_1 \in \mathbb{C}$. Suppose also that the sequence $\{\mathcal{P}_n^*(k_2;x)\}_{n=0}^{\infty}$ of kernel polynomials generated by Christoffel transformation exists for some $k_2 \in \mathbb{C}$. Then there exist unique sequences of constants $\{\gamma_n\}$ and $\{\eta_n\}$ with an explicit expression such that the sequence of polynomials $\{{\mathcal{Q}}_n^{C}(k_1,k_2;x)\}$ given by
	\begin{align}{\label{Recovery from kernel+quasi-type kernel}}
		{\mathcal{Q}}_n^{C}(k_1,k_2;x): =\frac{x-k_1}{x-\gamma_{n-1}}\mathcal{T}_{n}^*(k_1;x)+\eta_{n-1}\frac{x-k_2}{x-\gamma_{n-1}}\mathcal{P}_{n-1
		}^*(k_2;x)
	\end{align}
satisfies the same orthogonality as that of $\{\mathcal{P}_n(x)\}$. In particular, if $k=k_1=k_2 \in \mathbb{C}$ then

\begin{equation*}
		\tilde{\mathcal{T}}_{n}^*(k;x)=\mathcal{P}_{n}^*(k;x)+\tilde{B}_{n}\mathcal{P}_{n-1}^*(k;x)=\frac{x-\gamma_{n-1}}{x-k}\mathcal{P}_{n}(x),
\end{equation*}
and if $supp(d\mu) \subset \mathbb{R}$ is compact then
\begin{equation*}
	\tilde{\mathcal{T}}_{n}^*(k;x) \in L^1(d\mu).
\end{equation*}
\end{theorem}


\begin{proof}If the sequence $\{\mathcal{Q}^C_n(k_1,k_2;x)\}$ is orthogonal with respect to the linear functional $\mathcal{L}$, then by uniqueness theorem of orthogonal polynomials with respect to linear functional, $\{\mathcal{Q}^C_n(k_1,k_2;x)\}$ and $\{\mathcal{P}_n(x)\}$ are the same system of orthogonal polynomials and vice-versa. Consider
	\begin{align*}
	& {\mathcal{Q}}_{n+1}^{C}(k_1,k_2;x) =\frac{x-k_1}{x-\gamma_{n}}\mathcal{T}_{n+1}^*(k_1;x)+\eta_{n}\frac{x-k_2}{x-\gamma_{n}}\mathcal{P}_{n
		}^*(k_2;x)\\
	&=\frac{1}{x-\gamma_{n}}\left[(x-k_1)\mathcal{T}_{n+1}^*(k_1;x)-(x-\gamma_{n})\mathcal{P}_{n+1}(x)+\eta_{n}(x-k_2)\mathcal{P}_{n
	}^*(k_2;x)\right]+\mathcal{P}_{n+1}(x).
	\end{align*}
	
	Using the definitions of kernel polynomials and quasi-type kernel polynomial of order one, we have
	 \begin{align*}
		&(x-k_1)\mathcal{T}_{n+1}^*(k_1;x)-(x-\gamma_n)\mathcal{P}_{n+1}(x)+ \eta_n(x-k_2)\mathcal{P}_n^*(k_2;x)\\
		&= (x-k_1)(\mathcal{P}_{n+1}^*(k_1;x) + B_{n+1}\mathcal{P}_{n}^*(k_1;x))-(x-\gamma_n)\mathcal{P}_{n+1}(x)+ \eta_n(x-k_2)\mathcal{P}_n^*(k_2;x)\\
		&= \mathcal{P}_{n+2}(x)- \frac{\mathcal{P}_{n+2}(k_1)}{\mathcal{P}_{n+1}(k_1)} \mathcal{P}_{n+1}(x) + B_{n+1}\mathcal{P}_{n+1}(x)-B_{n+1}\frac{\mathcal{P}_{n+1}(k_1)}{\mathcal{P}_{n}(k_1)} \mathcal{P}_{n}(x)- (x-\gamma_n)\mathcal{P}_{n+1}(x) \\
		&\hspace{9.9cm}+ \eta_n\mathcal{P}_{n+1}(x)-\eta_n \frac{\mathcal{P}_{n+1}(k_2)}{\mathcal{P}_{n}(k_2)} \mathcal{P}_{n}(x)
	\end{align*}

Combining the coefficients of $\mathcal{P}_{n+2}(x)$, $\mathcal{P}_{n+1}(x)$ and $\mathcal{P}_{n}(x)$, we can write the above expression as
\begin{align}\label{5}
\nonumber	&(x-k_1)\mathcal{T}_{n+1}^*(k_1;x)-(x-\gamma_n)\mathcal{P}_{n+1}(x)+ \eta_n(x-k_2)\mathcal{P}_n^*(k_2;x)=\mathcal{P}_{n+2}(x)\\
&- \left(x-\gamma_n+ \frac{\mathcal{P}_{n+2}(k_1)}{\mathcal{P}_{n+1}(k_1)}-B_{n+1}- \eta_n\right)\mathcal{P}_{n+1}(x)- \left(\eta_n \frac{\mathcal{P}_{n+1}(k_2)}{\mathcal{P}_{n}(k_2)} + B_{n+1}\frac{\mathcal{P}_{n+1}(k_1)}{\mathcal{P}_{n}(k_1)}\right)\mathcal{P}_n(x).
\end{align}
Consider
\begin{align*}
\eta_n= -\left(\lambda_{n+2} + B_{n+1}\frac{\mathcal{P}_{n+1}(k_1)}{\mathcal{P}_{n}(k_1)} \right)\frac{\mathcal{P}_{n}(k_2)}{\mathcal{P}_{n+1}(k_2)},
\end{align*}
and
\begin{align*}
\gamma_n =c_{n+2}+ \frac{\mathcal{P}_{n+2}(k_1)}{\mathcal{P}_{n+1}(k_1)}- B_{n+1}
+ \left(\lambda_{n+2} + B_{n+1}\frac{\mathcal{P}_{n+1}(k_1)}{\mathcal{P}_{n}(k_1)} \right)\frac{\mathcal{P}_{n}(k_2)}{\mathcal{P}_{n+1}(k_2)}.
\end{align*}Then, we  can write  \eqref{5} as
		\newline
		$\displaystyle(x-k_1)\mathcal{T}_{n+1}^*(k_1;x)-(x-\gamma_n)\mathcal{P}_{n+1}(x)+ \eta_n(x-k_2)\mathcal{P}_n^*(k_2;x)$
\begin{align}\label{6}
=\mathcal{P}_{n+2}(x)- \left(x-c_{n+2}\right)\mathcal{P}_{n+1}(x)+ \lambda_{n+2}\mathcal{P}_n(x),
\end{align}
		where
		\begin{equation*}
		c_{n+2}=\gamma_n- \frac{\mathcal{P}_{n+2}(k_1)}{\mathcal{P}_{n+1}(k_1)}+B_{n+1}+ \eta_n, ~\lambda_{n+2}=-\eta_n \frac{\mathcal{P}_{n+1}(k_2)}{\mathcal{P}_{n}(k_2)} - B_{n+1}\frac{\mathcal{P}_{n+1}(k_1)}{\mathcal{P}_{n}(k_1)}.
		\end{equation*}\\
	The above expression \eqref{6} must be equal to zero because $\mathcal{P}_{n}(x)$ is a monic orthogonal polynomial sequence with respect to measure $d\mu$. Hence, by Favard's theorem it satisfies the TTRR, which gives the desired result.
If both quasi-type kernel polynomial of order one and kernel polynomials exist for some $k=k_1=k_2\in \mathbb{C}$, then \eqref{Recovery from kernel+quasi-type kernel} can be written as
\begin{align*}
(x-k)\left(\mathcal{P}_{n+1}^*(k;x)+\tilde{B}_{n+1}\mathcal{P}_{n}^*(k;x)\right)-(x-\gamma_n)\mathcal{P}_{n+1}(x)=0,
\end{align*}	
where $\tilde{B}_{n+1}=B_{n+1}+\eta_n$.
\noindent \text{This implies}~~ $(x-k)\tilde{\mathcal{T}}_{n+1}^*(k;x)-(x-\gamma_n)\mathcal{P}_{n+1}(x)=0$, which further gives

\begin{equation}
		\tilde{\mathcal{T}}_{n+1}^*(k;x)=\frac{x-\gamma_n}{x-k}\mathcal{P}_{n+1}(x).
\end{equation}
If the support of a measure $\mu$ is a compact subset of  real line and $k \not \in supp(d\mu)$ then
\begin{align*}
	\lVert \tilde{\mathcal{T}}_{n+1}^*(k;x) \rVert_{L^1(d\mu)} &= \int \left\vert \frac{x-\gamma_n}{x-k}\mathcal{P}_{n+1}(x)\right\vert d\mu \\
	&\leq \int \frac{x}{x-k}\mathcal{P}_{n+1}(x)d\mu + \gamma_{n} \int \frac{\mathcal{P}_{n+1}(x)}{x-k}d\mu \\
	&\leq \left(\int \frac{1}{\vert x-k\vert^2}d\mu\right)^{1/2}\left[\left(\int \vert x\mathcal{P}_{n+1}\vert^2d\mu\right)^{1/2} +\gamma_{n}\left(\int \vert \mathcal{P}_{n+1}\vert^2d\mu\right)^{1/2}\right]\\
	&< \infty.
\end{align*}
In the above, we used the triangle inequality and H\"older's inequality to obtain the first and second inequalities, respectively. Moreover,  finiteness follows directly from the fact that multiplication by $x$ is in $L^2(d\mu)$ and $k\not\in \text{supp}(d\mu)$.
\end{proof}
\subsection{Geronimus transformation}
Let $\mathcal{L}$ be a linear functional. For given $k\in\mathbb{C}$, define
\begin{align*}
	\widetilde{\mathcal{L}}((x-k)p(x))= \mathcal{L}(p(x))
\end{align*}
for any polynomial $p(x)$. This transformation $\widetilde{\mathcal{L}}$ is known as \emph{Geronimus transformation} at $k\in\mathbb{C}$.  Geronimus transformation can be regarded  as the inverse of Christoffel transformation at $k$\cite{Castillo_2017_linear spectral_integral tranform}.\\
For any polynomial $p(x)$, we can write

\begin{align*}
	\widetilde{\mathcal{L}}(p(x))&= \widetilde{\mathcal{L}}\left(\left(\frac{p(x)-p(k)}{x-k}\right)(x-k) + p(k)\right)\\
	&= \mathcal{L}\left(\frac{p(x)-p(k)}{x-k}\right) + p(k)\widetilde{\mathcal{L}}(1),
\end{align*}
where $\widetilde{\mathcal{L}}(1)$ is not uniquely determined, and hence an arbitrary constant. However, $\widetilde{\mathcal{L}}(1)\neq0$, because there does not exist any orthogonal polynomial sequence with such property\cite{Derevyagin_Bailey_CJM by DT}.
\newline

%
Next we state the result in which  a sequence of quasi-orthogonal polynomials of order one with suitable choice of $A_n$ is taken, which forms an orthogonal polynomial sequence with respect to the  Geronimus transformation at $k$.

\begin{theorem}\rm{{\cite{Castillo_2017_linear spectral_integral tranform,Garcia_Paco_2022_JMAA}}}\label{Geronimus transformation theorem}
	Let  $\{\mathcal{P}_n(x)\}_{n=0}^{\infty}$ be the sequence of orthogonal polynomials with respect to the positive definite linear functional $\mathcal{L}$. If $k\in \mathbb{C}\setminus \text{supp} \mu$ then the sequence of monic polynomials
	\begin{eqnarray}\label{Geronimus transformation expression}
		\widetilde{\mathcal{P}}_n(k;x)=\mathcal{P}_n(x) + A_n \mathcal{P}_{n-1}(x),
	\end{eqnarray}
	where
	\begin{align*}
		A_n=-\frac{\int\frac{\mathcal{P}_n(x)}{k-x}d\mu(x)}{\int\frac{\mathcal{P}_{n-1}(x)}{k-x}d\mu(x)}
	\end{align*}
	is an orthogonal polynomial sequence for the corresponding Geronimus transformation $\widetilde{\mathcal{L}}$ at $k$.
\end{theorem}
In Theorem \ref{Geronimus transformation theorem}, we see that one can find the explicit form of polynomials generated by Geronimus transformation in terms of orthogonal polynomials $\mathcal{P}_n(x)$. In the Proposition \ref{Poly intemrs poly gen by GT}, we give the expression for orthogonal polynomials $\mathcal{P}_n(x)$ in terms of polynomials generated by $\widetilde{\mathcal{L}}$ using TTRR \eqref{TTRR}. Note that the similar expression for $\mathcal{P}_n(x)$ with different approach was given in \cite{Garcia_Paco_2022_JMAA} and references therein.
\begin{proposition}\label{Poly intemrs poly gen by GT}
	Let  $\{\widetilde{\mathcal{P}}_n(x)\}_{n=0}^{\infty}$ be the sequence of orthogonal polynomials with respect to the Geronimus transformation which exists for some $k\in \mathbb{C}$. Then we can write $\mathcal{P}_n(x)$ in terms of linear combinations of $\widetilde{\mathcal{P}}_n(k;x)$ and $\widetilde{\mathcal{P}}_{n+1}(k;x)$  as follows:
	\begin{align*}
		\mathcal{P}_n(x)=\frac{1}{x-k}\widetilde{\mathcal{P}}_{n+1}(k;x)-\frac{1}{x-k}\frac{\lambda_{n+1}}{A_n}\widetilde{\mathcal{P}}_{n}(k;x).
	\end{align*}
\end{proposition}

\begin{proof}
	Using equation \eqref{Geronimus transformation expression}, we can write
	\begin{align*}
		&\displaystyle\frac{1}{x-k}\widetilde{\mathcal{P}}_{n+1}(k;x)+\frac{1}{x-k}B_n\widetilde{\mathcal{P}}_{n}(k;x)\\
		&\hspace{2.5cm}=\frac{1}{x-k}\left[\mathcal{P}_{n+1}(x)-(x-k-A_{n+1}+B_n)\mathcal{P}_{n}(x)-B_nA_n\mathcal{P}_{n-1}(x)\right]+\mathcal{P}_{n}(x).
	\end{align*}
	Since $A_n\neq 0$, by putting
	\begin{eqnarray*}
		B_n=-\frac{\lambda_{n+1}}{A_n},
	\end{eqnarray*}
	we can write the above equation as
	\begin{align*}\displaystyle
		&\frac{1}{x-k}\widetilde{\mathcal{P}}_{n+1}(k;x)+\frac{1}{x-k}B_n\widetilde{\mathcal{P}}_{n}(k;x)\\
		&\hspace{3cm}=\frac{1}{x-k}\left[\mathcal{P}_{n+1}(x)-(x-c_{n+1})\mathcal{P}_{n}(x)+\lambda_{n+1}\mathcal{P}_{n-1}(x)\right]+\mathcal{P}_{n}(x)
	\end{align*}
	where
	\begin{align*}
		c_{n+1}=k+A_{n+1}-B_n, ~~\lambda_{n+1}=-B_nA_n.
	\end{align*}
	Using TTRR \eqref{TTRR}, we get the desired result.	
\end{proof}

 In the next theorem, we recover the orthogonality for ${\mathcal{Q}}_n^{G}(k_1,k_2;x)$   by obtaining three sequences of parameters.

\begin{theorem}\label{QK+OP+Geronimus}
	Let $\{\mathcal{P}_n(x)\}_{n=0}^{\infty}$ be a monic orthogonal polynomial sequence with respect to the positive definite linear functional $\mathcal{L}$. Let $\mathcal{T}_{n}^*(k_2,x)$ be a quasi-type kernel polynomial of order one for some $k_2 \in \mathbb{C}$. Further, suppose that the sequence $\{	\widetilde{\mathcal{P}}_{n}(k_1;x)\}_{n=0}^{\infty}$ of the polynomials corresponding to Geronimus transformation exist for some $k_1 \in \mathbb{C}$. Then there exist unique sequences of constants $\{\alpha_n\},\{\gamma_n\}$ and $\{\eta_n\}$ such that the sequence of polynomials $\{{\mathcal{Q}}_n^{G}(k_1,k_2;x)\}$ given by
	\begin{align}\label{Recovery from quasi-type kernel+Geronimus}
		{\mathcal{Q}}_n^{G}(k_1,k_2;x) : =\frac{1}{\alpha_nx-\gamma_n}\widetilde{\mathcal{P}}_{n+1}(k_1;x)+\eta_n\frac{x-k_2}{\alpha_nx-\gamma_n}\mathcal{T}_n^*(k_2;x)
	\end{align}
satisfies the same orthogonality as that of $\{{\mathcal{P}}_n(x)\}$.
\end{theorem}

\begin{proof}If the sequence $\{\mathcal{Q}^G_n(k_1,k_2;x)\}$ is orthogonal with respect to the linear functional $\mathcal{L}$, then by uniqueness theorem of orthogonal polynomials, $\{\mathcal{Q}^G_n(k_1,k_2;x)\}$ and $\{\mathcal{P}_n(x)\}$ are the same system of orthogonal polynomials and vice-versa. We can write \eqref{Recovery from quasi-type kernel+Geronimus} as
	\begin{align*}
		&{\mathcal{Q}}_n^{G}(k_1,k_2;x) =\frac{1}{\alpha_nx-\gamma_n}\widetilde{\mathcal{P}}_{n+1}(k_1;x)+\eta_n\frac{x-k_2}{\alpha_nx-\gamma_n}\mathcal{T}_n^*(k_2;x)\\
		&\hspace{2.2cm}=\frac{1}{\alpha_nx-\gamma_n}\left[\widetilde{\mathcal{P}}_{n+1}(k_1;x)-(\alpha_{n}x-\gamma_{n})\mathcal{P}_{n}(x)+\eta_n(x-k_2)\mathcal{T}_n^*(k_2;x)\right]+\mathcal{P}_{n}(x).
	\end{align*}
	Considering \eqref{Geronimus transformation expression} together with the definition of kernel polynomials and quasi-type kernel polynomial of order one gives
\begin{align*}
	&\widetilde{\mathcal{P}}_{n+1}(k_1;x)-(\alpha_nx-\gamma_n)\mathcal{P}_n(x)+ \eta_n(x-k_2)\mathcal{T}_n^*(k_2;x)\\
	&= \mathcal{P}_{n+1}(x) + A_{n+1} \mathcal{P}_{n}(x)-(\alpha_nx-\gamma_n)\mathcal{P}_n(x)+\eta_n(x-k_2)( \mathcal{P}_n^*(k_2;x) + \tilde{B}_{n}\mathcal{P}_{n-1}^*(k_2;x))\\
	&= \mathcal{P}_{n+1}(x) + A_{n+1} \mathcal{P}_{n}(x)-(\alpha_nx-\gamma_n)\mathcal{P}_n(x)+ \eta_n\mathcal{P}_{n+1}(x)-\eta_n \frac{\mathcal{P}_{n+1}(k_2)}{\mathcal{P}_{n}(k_2)} \mathcal{P}_{n}(x)\\ &\hspace{6cm}+\eta_n\tilde{B}_{n}\mathcal{P}_{n}(x)-\eta_n\tilde{B}_{n}\frac{\mathcal{P}_{n}(k_2)}{\mathcal{P}_{n-1}(k_2)} \mathcal{P}_{n-1}(x).
\end{align*}
	  Combining the coefficients of $\mathcal{P}_{n+1}(x)$, $\mathcal{P}_{n}(x)$ and $\mathcal{P}_{n-1}(x)$, we get
	  \newline
	  $\displaystyle \widetilde{\mathcal{P}}_{n+1}(k_1;x)-(\alpha_nx-\gamma_n)\mathcal{P}_n(x)+ \eta_n(x-k_2)\mathcal{T}_n^*(k_2;x)$
\begin{align}\label{4}
	&\nonumber=(1+\eta_n)\left[\mathcal{P}_{n+1}(x)-\left(\frac{\alpha_n}{1+\eta_n}x-\frac{\gamma_n +A_{n+1}-\eta_n \frac{\mathcal{P}_{n+1}(k_2)}{\mathcal{P}_{n}(k_2)}+\eta_n\tilde{B}_n}{1+\eta_n}\right)\mathcal{P}_n(x)\right. \\
	&\hspace{9cm}\left.-\frac{\eta_n\tilde{B}_{n}\mathcal{P}_{n}(k_2)}{(1+\eta_n)\mathcal{P}_{n-1}(k_2)} \mathcal{P}_{n-1}(x)\right].
\end{align}
	Since $\mathcal{P}_{n}(k_2) \neq 0,\mathcal{P}_{n-1}(k_2)\neq 0$,  by substituting
\begin{align*}
	\eta_n=-\frac{\lambda_{n+1}}{\lambda_{n+1}+\tilde{B}_n\frac{\mathcal{P}_{n}(k_2)}{\mathcal{P}_{n-1}(k_2)}}, \quad\alpha_n=1-\frac{\lambda_{n+1}}{\lambda_{n+1}+\tilde{B}_n\frac{\mathcal{P}_{n}(k_2)}{\mathcal{P}_{n-1}(k_2)}},
\end{align*}
and
\begin{align*}
	\gamma_{n}=c_{n+1}(1+\eta_n)-A_{n+1}+\eta_n\frac{\mathcal{P}_{n+1}(k_2)}{\mathcal{P}_{n}(k_2)}-\eta_n \tilde{B}_n,
\end{align*}
     we can write the right side of the  expression \eqref{4} as
\begin{equation}\label{use}
	  (1+\eta_n)\left[\mathcal{P}_{n+1}(x)- \left(x-c_{n+1}\right)\mathcal{P}_{n}(x)+ \lambda_{n+1}\mathcal{P}_{n-1}(x)\right],
\end{equation}
      where
\begin{align*}
	 c_{n+1}=\frac{\gamma_n +A_{n+1}-\eta_n \frac{\mathcal{P}_{n+1}(k_2)}{\mathcal{P}_{n}(k_2)}+\eta_n\tilde{B}_n}{1+\eta_n}~ \text{and}~\lambda_{n+1}=-\frac{\eta_n\tilde{B}_{n}\mathcal{P}_{n}(k_2)}{(1+\eta_n)\mathcal{P}_{n-1}(k_2)}.
\end{align*}
	The above expression \eqref{use} must be equal to zero. Since $\mathcal{P}_{n}(x)$ is a monic orthogonal polynomial sequence, by Favard's theorem it satisfies TTRR. This completes the proof.		
\end{proof}

\subsection{Uvarov Transformation}
Linear spectral transformations play a significant role in the study of perturbation of orthogonal polynomials. We can obtain one of the main transformations by adding point mass to the original measure. In other words, if $\mathcal{L}$ is a quasi-definite linear functional, then we can define $\hat{\mathcal{L}}$ by
\begin{align*}
\hat{\mathcal{L}}=\mathcal{L}+R_o \delta(x-k),
\end{align*}
where $\delta(\cdot)$ is a mass point at $k$ and $R_o$ is a non zero constant. The new linear functional $\hat{\mathcal{L}}$ is known as canonical Uvarov transformation\cite{Uvarov trans. first paper} of $\mathcal{L}$.\\

 To study the structure of polynomials corresponding to Uvarov transformation, it is essential that the Uvarov transformation has at least the property of quasi definiteness. In this regard, the necessary and sufficient conditions for preserving the quasi definite property of the linear functional are given in \cite{M.M} . In addition, the condition for preserving the positive definite property of Uvarov transformation from the original positive definite linear functional is given in \cite{Positive definite of Uvarov}.

 \begin{theorem}[cf. {\cite[page 256]{Paco_2018_conference}}]
 	Let $\{\hat{\mathcal{P}}_n(x)\}_{n=0}^{\infty}$ be a monic orthogonal polynomial sequence corresponding to the quasi definite linear functional $\hat{\mathcal{L}}$.  Suppose that the sequence $\{\mathcal{P}_n^*(k;x)\}_{n=0}^{\infty}$ of kernel polynomials generated by Christoffel transformation exists for some $k \in \mathbb{C}$. Then we have
 	\begin{align}
 		\hat{\mathcal{P}}_n(x)=\mathcal{P}_n(x)-T_n\mathcal{P}^*_{n-1}(k;x),
 	\end{align}
 	where
 	\begin{align*}
 		T_n=\frac{R_o\mathcal{P}_n^2(k)}{\lambda_1...\lambda_{n+1}\left(1+\frac{R_o\mathcal{P}_{n-1}^*(k;k)\mathcal{P}_n(k)}{\lambda_1...\lambda_{n+1}}\right)}.
 	\end{align*}
 \end{theorem}
The following result shows that one can recover the original sequence of orthogonal polynomials from the linear combination of quasi-type kernel polynomials of order one and polynomials generated by Uvarov transformation with rational coefficients and by suitably identifying three sequences of constants.

\begin{theorem}\label{QK+OP+Uvarov}
 Let $\mathcal{T}_{n}^*(k_2,x)$ be a quasi-type kernel polynomial of order one for some $k_2 \in \mathbb{C}$. Further, suppose that sequence $\{\hat{\mathcal{P}}_{n}(x)\}_{n=0}^{\infty}$ of the polynomials corresponding to Uvarov transformation. Then there exist unique sequences of constants $\{\alpha_n\},\{\gamma_n\}$ and $\{\eta_n\}$ such that the sequence of polynomials $\{{\mathcal{Q}}_n^{U}(k_1,k_2;x)\}$ given by
\begin{align}\label{Recovery from quasi-type kernel+Uvarov}
	{\mathcal{Q}}_n^{U}(k_1,k_2;x): =\frac{x-k_1}{\alpha_nx-\gamma_n}\hat{\mathcal{P}}_{n}(x)+\eta_n\frac{x-k_2}{\alpha_nx-\gamma_n}\mathcal{T}_n^*(k_2;x)
\end{align}	
satisfies the same orthogonality given by $\{{\mathcal{P}}_n(x)\}$.
\end{theorem}
\begin{proof}If the sequence $\{\mathcal{Q}^U_n(k_1,k_2;x)\}$ is orthogonal with respect to the linear functional $\mathcal{L}$, then by uniqueness theorem of orthogonal polynomials, $\{\mathcal{Q}^U_n(k_1,k_2;x)\}$ and $\{\mathcal{P}_n(x)\}$ are the same system of orthogonal polynomials and vice-versa. We can write the expression \eqref{Recovery from quasi-type kernel+Uvarov} as
	\begin{align*}
		&{\mathcal{Q}}_n^{U}(k_1,k_2;x) =\frac{x-k_1}{\alpha_nx-\gamma_n}\hat{\mathcal{P}}_{n}(x)+\eta_n\frac{x-k_2}{\alpha_nx-\gamma_n}\mathcal{T}_n^*(k_2;x)\\
		&\hspace{1.5cm}=\frac{1}{\alpha_nx-\gamma_n}\left[(x-k_1)\hat{\mathcal{P}}_{n}(x)-(\alpha_nx-\gamma_n)\mathcal{P}_n(x)+\eta_n(x-k_2)\mathcal{T}_n^*(k_2;x)\right]+\mathcal{P}_n(x).
	\end{align*} First, we simplify the bracketed portion of the above equation. For this, consider
	\begin{align}
	&\nonumber(x-k_1)\hat{\mathcal{P}}_{n}(x)-(\alpha_nx-\gamma_n)\mathcal{P}_n(x)+ \eta_n(x-k_2)\mathcal{T}_n^*(k_2;x)\\
	&\nonumber=(x-k_1)\mathcal{P}_n(x)-T_n\mathcal{P}_n(x)+T_n\frac{\mathcal{P}_n(k_1)}{\mathcal{P}_{n-1}(k_1)}\mathcal{P}_{n-1}(x)-\alpha_nx\mathcal{P}_n(x)+\beta_n\mathcal{P}_n(x)\\
	&\nonumber\hspace{5.1cm}+\eta_n(x-k_2)\mathcal{P}_n^*(k_2;x)+\tilde{B}_n\eta_n(x-k_2)\mathcal{P}_{n-1}^*(k_2;x)\\
	&\nonumber=(1-\alpha_n)x\mathcal{P}_n(x)+(\beta_n-k_1-T_n)\mathcal{P}_n(x)+T_n\frac{\mathcal{P}_n(k_1)}{\mathcal{P}_{n-1}(k_1)}\mathcal{P}_{n-1}(x)+\eta_n\mathcal{P}_{n+1}(x)\\
	&\nonumber\hspace{3.3cm}-\eta_n\frac{\mathcal{P}_{n+1}(k_2)}{\mathcal{P}_{n}(k_2)}\mathcal{P}_{n}(x)+\eta_n\tilde{B}_n\mathcal{P}_n(x)-\eta_n\tilde{B}n\frac{\mathcal{P}_{n}(k_2)}{\mathcal{P}_{n-1}(k_2)}\mathcal{P}_{n-1}(x)\\
	&\nonumber=(1-\alpha_n)\left[\mathcal{P}_{n+1}(x)+c_{n+1}\mathcal{P}_{n}(x)+ \lambda_{n+1}\mathcal{P}_{n-1}(x)\right]+(\beta_n-k_1-T_n)\mathcal{P}_n(x)\\
	&\nonumber\hspace{3.6cm}+T_n\frac{\mathcal{P}_n(k_1)}{\mathcal{P}_{n-1}(k_1)}\mathcal{P}_{n-1}(x)+\eta_n\mathcal{P}_{n+1}(x)-\eta_n\frac{\mathcal{P}_{n+1}(k_2)}{\mathcal{P}_{n}(k_2)}\mathcal{P}_{n}(x)\\
	&\nonumber\hspace{3.6cm}+\eta_n\tilde{B}_n\mathcal{P}_n(x)-\eta_n\tilde{B}n\frac{\mathcal{P}_{n}(k_2)}{\mathcal{P}_{n-1}(k_2)}\mathcal{P}_{n-1}(x)\\
	&\nonumber=(1-\alpha_n+\eta_n)\mathcal{P}_{n+1}(x)+\left(\beta_n-k_1-T_n-\eta_n\frac{\mathcal{P}_{n+1}(k_2)}{\mathcal{P}_n(k_2)}+\eta_n\tilde{B}_n+c_{n+1}-\alpha_nc_{n+1}\right)\mathcal{P}_n(x)\\
	&\label{L.I}\hspace{3cm}+\left(T_n\frac{\mathcal{P}_n(k_1)}{\mathcal{P}_{n-1}(k_1)}-\eta_n\tilde{B}n\frac{\mathcal{P}_{n}(k_2)}{\mathcal{P}_{n-1}(k_2)}+\lambda_{n+1}-\alpha_n\lambda_{n+1}\right)\mathcal{P}_{n-1}(x).
	\end{align}
	 Here to obtain \eqref{L.I}, we have used the expression for multiplication by $x$ with $\mathcal{P}_n(x)$ and combining the coefficients of $\mathcal{P}_{n+1}(x),\mathcal{P}_{n}(x) $ and $\mathcal{P}_{n-1}(x)$ to obtain equation \eqref{L.I}. Next, setting the right side  of \eqref{L.I} equal to zero and by using the fact that $\mathcal{P}_{n+1}(x),\mathcal{P}_{n}(x) $ and $\mathcal{P}_{n-1}(x)$ are linearly independent, we get that the coefficients must be equal to zero. So we obtain the unique sequence of constants  $\{\alpha_n\},\{\gamma_n\}$ and $\{\eta_n\}$ as
\begin{equation*}
	\begin{split}
	\alpha_n&=1+\frac{T_n\frac{\mathcal{P}_{n}(k_1)}{\mathcal{P}_{n-1}(k_1)}}{\tilde{B}_n\frac{\mathcal{P}_{n}(k_2)}{\mathcal{P}_{n-1}(k_2)}+\lambda_{n+1}}\\
	\beta_n&=k_1+T_n+\left(c_{n+1}-\tilde{B}_n+\frac{\mathcal{P}_{n+1}(k_2)}{\mathcal{P}_n(k_2)}\right)\left(\frac{T_n\frac{\mathcal{P}_{n}(k_1)}{\mathcal{P}_{n-1}(k_1)}}{\tilde{B}_n\frac{\mathcal{P}_{n}(k_2)}{\mathcal{P}_{n-1}(k_2)}+\lambda_{n+1}}\right)\\
	\eta_n&=\frac{T_n\frac{\mathcal{P}_{n}(k_1)}{\mathcal{P}_{n-1}(k_1)}}{\tilde{B}_n\frac{\mathcal{P}_{n}(k_2)}{\mathcal{P}_{n-1}(k_2)}+\lambda_{n+1}},
	\end{split}
\end{equation*}
 this completes the proof.
\end{proof}
\subsection{Quasi-type kernel polynomials of order two}
Now we define the monic quasi-orthogonal polynomials of order two. Define $\mathcal{S}_{n}(x)$\cite{Derevyagin_Bailey_CJM by DT} as follows:
\begin{align}
	\mathcal{S}_n (x)= \mathcal{P}_n(x) + L_{n}\mathcal{P}_{n-1}(x)+ M_{n}\mathcal{P}_{n-2}(x),
\end{align}
where $\{\mathcal{P}_n(x)\}_{n=0}^{\infty}$ is a monic orthogonal polynomial sequence with respect to the  linear functional $\mathcal{L}$ for any choice of $L_n, M_n\in \mathbb{C}$.\\

If $\mathcal{L}(x^m\mathcal{S}_n(x))=0 ~~\text{for}~~ m= 0, 1,2,...,n-3, ~~\text{for any choice  of} ~~L_n, M_n\in \mathbb{C}$, then $\mathcal{S}_n(x)$ is called quasi-orthogonal polynomial of order two.\\

Similarly, we can extend this definition to the quasi-type kernel polynomial of order two with respect to $\mathcal{L}^*$. Define $\mathcal{S}_n^*(k;x)$ as follows:
\begin{align*}
\mathcal{S}_n^*(k;x)= \mathcal{P}_n^*(k;x) + \tilde{L}_{n}\mathcal{P}_{n-1}^*(k;x)+\tilde{M}_{n}\mathcal{P}_{n-2}^*(k;x),
\end{align*}
where $\{\mathcal{P}_n^*(k;x)\}_{n=0}^{\infty}$ is a sequence of kernel polynomials which exist for some $k\in \mathbb{C}$ and form a monic orthogonal polynomial system with respect to the quasi-definite linear functional $\mathcal{L}^*.$\\

If $\mathcal{L}^*(x^m\mathcal{S}_n^*(k;x))=0 ~~\text{for}~~ m= 0, 1,2,...,n-3 ~~\text{and for any choice  of} ~~\tilde{L}_n, \tilde{M}_n\in \mathbb{C}$, then $\mathcal{S}_n^*(k;x)$ is called quasi-type kernel polynomial of order two.\\

 In the next theorem, we recover the polynomials $\mathcal{P}_{n}(x)$ from the linear combination of iterated kernel polynomials\cite[page 9]{Derevyagin_Bailey_CJM by DT} with two parameters and quasi-type kernel polynomials of order two with rational coefficients. We obtain two sequences of constants responsible for obtaining $\mathcal{P}_{n}(x)$.

\begin{theorem}\label{Recovery from iterated kernel and quasi-type kernel}
Let $\{\mathcal{P}_n(x)\}_{n=0}^{\infty}$ be a monic orthogonal polynomial sequence with respect to the positive definite linear functional $\mathcal{L}$. Let $\mathcal{S}_{n+1}^*(k_1,x)$ be a quasi-type kernel polynomial of order two for some $k_1 \in \mathbb{C}$ with suitable choice of $\tilde{L}_n,\tilde{M}_n$ which satisfy
\begin{align}\label{condition coeff. of quasi-type kernel order 2}
\tilde{L}_n+\tilde{M_n} \frac{\mathcal{P}_n(k_1)}{\lambda_{n+1}\mathcal{P}_{n-1}(k_1)}= \frac{\mathcal{P}_{n+2}(k_1)}{\mathcal{P}_{n+1}(k_1)}-\frac{\mathcal{P}_{n+2}(k_2)}{\mathcal{P}_{n+1}(k_2)}-\frac{\mathcal{P}_{n+1}^{*}(k_2,k_3)}{\mathcal{P}_{n}^{*}(k_2,k_3)}.
\end{align} Further, suppose that the sequence $\{\mathcal{P}_{n}^{**}(k_2,k_3;x)\}_{n=0}^{\infty}$ of iterated kernel polynomials exists for some $k_2\in \mathbb{C}_{\mp}, k_3 \in \mathbb{C}_{\pm}$. Then, there exist unique sequences of constants $\{\alpha_n\}, \{\beta_n\}$ such that the sequence of polynomials $\{{\mathcal{Q}}_n^{S}(k_1,k_2,k_3;x)\}$ given by
\begin{align}{\label{Recovery from quasi-type kernel order 2+iterated kernel}}
{\mathcal{Q}}_n^{S}(k_1,k_2,k_3;x): =\frac{x-k_1}{\alpha_n x-\beta_n}\mathcal{S}_{n+1}^*(k_1;x)-\frac{(x-k_2)(x-k_3)}{\alpha_n x-\beta_n}\mathcal{P}_n^*(k_2,k_3;x)
\end{align}
satisfies the same orthogonality given by the polynomials $\{{\mathcal{P}}_n(x)\}$.
\end{theorem}
\begin{remark}
	The iterated kernel polynomials sequence  for some $k_2\in \mathbb{C}_{\mp}, k_3 \in \mathbb{C}_{\pm}$ are given in \rm\cite[eq. 3.5]{Derevyagin_Bailey_CJM by DT}, where $\mathbb{C}_{\mp}=\{z:\text{Im}z\lessgtr0\}.$
\end{remark}

\begin{proof}If the sequence $\{\mathcal{Q}^S_n(k_1,k_2,k_3;x)\}$ is orthogonal with respect to the linear functional $\mathcal{L}$, then by uniqueness theorem of orthogonal polynomials, $\{\mathcal{Q}^S_n(k_1,k_2,k_3;x)\}$ and $\{\mathcal{P}_n(x)\}$ are the same system of orthogonal polynomials and vice-versa. We can write the expression \eqref{Recovery from quasi-type kernel order 2+iterated kernel} as
	\begin{align*}
		&{\mathcal{Q}}_n^{S}(k_1,k_2,k_3;x)=\frac{x-k_1}{\alpha_n x-\beta_n}\mathcal{S}_{n+1}^*(k_1;x)-\frac{(x-k_2)(x-k_3)}{\alpha_n x-\beta_n}\mathcal{P}_n^*(k_2,k_3;x)\\
		&=\frac{1}{\alpha_n x-\beta_n}\left[(x-k_1)\mathcal{S}_{n+1}^*(k_1;x)-(\alpha_n x-\beta_n)\mathcal{P}_n(x)-(x-k_2)(x-k_3)\mathcal{P}_n^*(k_2,k_3;x)\right]\\
		&\hspace{12cm}+\mathcal{P}_n(x).
	\end{align*}
	
	Using the definition of quasi-type kernel polynomial of order two and the expression of kernel and iterated kernel polynomials, we have
	\newline
	$\displaystyle (x-k_1)\mathcal{S}_{n+1}^*(k_1;x)-(\alpha_n x-\beta_n)\mathcal{P}_{n}(x)-(x-k_2)(x-k_3)\mathcal{P}_n^*(k_2,k_3;x)$
	\begin{align*}
	&=\mathcal{P}_{n+2}(x)-\frac{\mathcal{P}_{n+2}(k_1)}{\mathcal{P}_{n+1}(k_1)}\mathcal{P}_{n+1}(x)+\tilde{L}_n\mathcal{P}_{n+1}(x)-\tilde{L}_n\frac{\mathcal{P}_{n+1}(k_1)}{\mathcal{P}_{n}(k_1)}\mathcal{P}_{n}(x)+\tilde{M_n}\mathcal{P}_{n}(x)\\
	&-\tilde{M_n}\frac{\mathcal{P}_{n}(k_1)}{\mathcal{P}_{n-1}(k_1)}\mathcal{P}_{n-1}(x)-\alpha_n x\mathcal{P}_{n}(x)+\beta_n\mathcal{P}_{n}(x)-\mathcal{P}_{n+2}(x)+\frac{\mathcal{P}_{n+2}(k_2)}{\mathcal{P}_{n+1}(k_2)}\mathcal{P}_{n+1}(x)\\
	&\hspace{4cm}+\frac{\mathcal{P}_{n+1}^{*}(k_2,k_3)}{\mathcal{P}_{n}^{*}(k_2,k_3)}\mathcal{P}_{n+1}(x)-\frac{\mathcal{P}_{n+1}^{*}(k_2,k_3)}{\mathcal{P}_{n}^{*}(k_2,k_3)}\frac{\mathcal{P}_{n+1}(k_2)}{\mathcal{P}_{n}(k_2)}\mathcal{P}_{n}(x).
	\end{align*}

Since $\{\mathcal{P}_n(x)\}_{n=0}^{\infty}$ is a monic orthogonal polynomial sequence, we can use \eqref{TTRR} to write the expression for $x\mathcal{P}_n(x)$, which gives
\newline 	$\displaystyle (x-k_1)\mathcal{S}_{n+1}^*(k_1;x)-(\alpha_n x-\beta_n)\mathcal{P}_{n}(x)-(x-k_2)(x-k_3)\mathcal{P}_n^*(k_2,k_3;x)$
	\begin{align*}
		&=-\frac{\mathcal{P}_{n+2}(k_1)}{\mathcal{P}_{n+1}(k_1)}\mathcal{P}_{n+1}(x)+\tilde{L}_n\mathcal{P}_{n+1}(x)-\tilde{L}_n\frac{\mathcal{P}_{n+1}(k_1)}{\mathcal{P}_{n}(k_1)}\mathcal{P}_{n}(x)+\tilde{M_n}\mathcal{P}_{n}(x)\\
		&\hspace{.5cm}-\tilde{M_n}\frac{\mathcal{P}_{n}(k_1)}{\mathcal{P}_{n-1}(k_1)}\mathcal{P}_{n-1}(x)-\alpha_n \left[\mathcal{P}_{n+1}(x)+c_{n+1}\mathcal{P}_{n}(x)+ \lambda_{n+1}\mathcal{P}_{n-1}(x)\right]+\beta_n\mathcal{P}_{n}(x)\\
		&\hspace{.5cm}+\frac{\mathcal{P}_{n+2}(k_2)}{\mathcal{P}_{n+1}(k_2)}\mathcal{P}_{n+1}(x)+\frac{\mathcal{P}_{n+1}^{*}(k_2,k_3)}{\mathcal{P}_{n}^{*}(k_2,k_3)}\mathcal{P}_{n+1}(x)-\frac{\mathcal{P}_{n+1}^{*}(k_2,k_3)}{\mathcal{P}_{n}^{*}(k_2,k_3)}\frac{\mathcal{P}_{n+1}(k_2)}{\mathcal{P}_{n}(k_2)}\mathcal{P}_{n}(x).
	\end{align*}
	Combining the coefficients of $\mathcal{P}_{n+1}(x)$, $\mathcal{P}_{n}(x)$ and $\mathcal{P}_{n-1}(x)$, we get
	\newline    $\displaystyle (x-k_1)\mathcal{S}_{n+1}^*(k_1;x)-(\alpha_n x-\beta_n)\mathcal{P}_{n}(x)-(x-k_2)(x-k_3)\mathcal{P}_n^*(k_2,k_3;x)$
	\begin{align*}
&\hspace{0.5cm}=\left(-\frac{\mathcal{P}_{n+2}(k_1)}{\mathcal{P}_{n+1}(k_1)}+\frac{\mathcal{P}_{n+2}(k_2)}{\mathcal{P}_{n+1}(k_2)}+\tilde{L}_n-\alpha_n+\frac{\mathcal{P}_{n+1}^{*}(k_2,k_3)}{\mathcal{P}_{n}^{*}(k_2,k_3)}\right)\mathcal{P}_{n+1}(x)\\
&\hspace{0.5cm}+\left(-\tilde{M}_n\frac{\mathcal{P}_{n}(k_1)}{\mathcal{P}_{n-1}(k_1)}- \alpha_n \lambda_{n+1}\right)\mathcal{P}_{n-1}(x)+\left(-\tilde{L}_n\frac{\mathcal{P}_{n+1}(k_1)}{\mathcal{P}_{n}(k_1)}+\tilde{M_n}+\beta_n-\alpha_nc_{n+1}\right.\\
&\hspace{0.5cm}\qquad\left.-\frac{\mathcal{P}_{n+1}^{*}(k_2,k_3)}{\mathcal{P}_{n}^{*}(k_2,k_3)}\frac{\mathcal{P}_{n+1}(k_2)}{\mathcal{P}_{n}(k_2)}\right)\mathcal{P}_{n}(x).
	\end{align*}
	
	Setting the left side of the above equation equal to zero and by using the fact that $\mathcal{P}_{n+1}(x),\mathcal{P}_{n}(x) $ and $\mathcal{P}_{n-1}(x)$ are linearly independent we get that the coefficients must be equal to zero. This gives
		\begin{equation*}
		\begin{split}
	\alpha_n &=-\frac{1}{\lambda_{n+1}}\tilde{M_n}\frac{\mathcal{P}_n(k_1)}{\mathcal{P}_{n-1}(k_1)},\\ \beta_n&= \tilde{L}_n\frac{\mathcal{P}_{n+1}(k_1)}{\mathcal{P}_{n}(k_1)}-\tilde{M_n}+\lambda_{n+2}\left(\frac{\mathlarger{\sum}\limits_{j=0}^{n+1} \frac{\mathcal{P}_{j}(k_3)\mathcal{P}_{j}(k_2)}{\lambda_1 \lambda_2...\lambda_{j+1}}}{\mathlarger{\sum}\limits_{j=0}^{n} \frac{\mathcal{P}_{j}(k_3)\mathcal{P}_{j}(k_2)}{\lambda_1 \lambda_2...\lambda_{j+1}}}\right)-\frac{1}{\lambda_{n+1}}\frac{\mathcal{P}_n(k_1)}{\mathcal{P}_{n-1}(k_1)}c_{n+1}.
	\end{split}
	\end{equation*}
	
%
	
 We used the Christoffel-Darboux formula\cite{Chihara book} and the fact that zeros\cite{Szego} of $\mathcal{P}_{n}(x)$ lie on the real line for $n=1,2,...$ , for $k_2,k_3 \in \mathbb{C}_{\pm}$ to write the expression for $\beta_n$. This completes the proof.
 \end{proof}
\section{Ratio of kernel polynomials and continued fractions}\label{sec:Ratio}
While considering the quasi-type kernel polynomials of order two, \eqref{condition coeff. of quasi-type kernel order 2} provides the ratio of iterated kernel polynomials. Further, in Example \ref{Kernel of Chebyshev and reverse}, we are interested to find the behavior of the ratio of Chebyshev polynomial and Chebyshev kernel polynomial. To answer the above problem, we need the ratio of kernel polynomials. In particular, in this section, we are interested in the limiting case of ratio of kernel polynomials, which is addressed in Theorem \ref{th:Kernel}. For this, we require the confluent form of Christoffel-Darboux formula which we recall in  Theorem \ref{Confluent CD}. Then we discuss the ratio of kernel polynomials in terms of infinite continued fractions.
As specific cases we exhibit the ratio of kernel of Laguerre polynomials and Jacobi polynomials, in terms of, Confluent and Gaussian hypergeometric functions, respectively.

\begin{theorem}\rm{\cite{Chihara book}}\label{Confluent CD}
	Let \{$\mathcal{P}_n(x)\}_{n=1}^{\infty}$ be a sequence of monic orthogonal polynomials and $\lambda_{n}\neq 0$. Then
	\begin{align*}
		\mathlarger{\sum}\limits_{j=0}^{n} \frac{ \mathcal{P}_{j}^2(x)}{\lambda_1 \lambda_2...\lambda_{j+1}}=\frac{\mathcal{P}_{n+1}'(x)\mathcal{P}_{n}(x)-\mathcal{P}_{n+1}(x)\mathcal{P}_{n}'(x)}{\lambda_1 \lambda_2...\lambda_{n+1}}\cdot
	\end{align*}
\end{theorem}
\noindent Now we will compute the ratio of kernel polynomials as $x$ approaches $k$.

\begin{theorem}\label{th:Kernel}
	Let $\{\mathcal{P}_n^*(k;x)\}_{n=0}^{\infty}$ be a sequence of kernel polynomials that exists for some $k\in \mathbb{C}$. Then
	\begin{align}\label{kernel ratio}
		\lim\limits_{x\rightarrow k}\frac{\mathcal{P}_{n+1}^*(k;x)}{\mathcal{P}_n^*(k;x)}=\frac{\mathcal{P}_n(k)}{\mathcal{P}_{n+1}(k)}\lambda_{n+2}\left(1+\frac{1}{\lambda_1 \lambda_2...\lambda_{n+2}}\frac{\mathcal{P}_{n+1}^2(k)}{\mathlarger{\sum}\limits_{j=0}^{n} \frac{ \mathcal{P}_{j}^2(k)}{\lambda_1 \lambda_2...\lambda_{j+1}}}\right)
	\end{align}
	and
	\begin{equation}\label{kernel ratio 2}
		\lim\limits_{x\rightarrow k}\frac{\mathcal{P}_{n}^*(k;x)}{\mathcal{P}_{n+1}^*(k;x)}=\frac{\mathcal{P}_{n+1}(k)}{\mathcal{P}_n(k)}\frac{1}{\lambda_{n+2}}\left(1-\frac{\mathcal{P}_{n+1}^2(k)}{\lambda_1 \lambda_2...\lambda_{n+2}\mathlarger{\sum}\limits_{j=0}^{n+1} \frac{ \mathcal{P}_{j}^2(k)}{\lambda_1 \lambda_2...\lambda_{j+1}}}\right).
	\end{equation}
\end{theorem}

\begin{proof}Using the Definition \ref{kernel} and Theorem \ref{Confluent CD}, we have
	\begin{align*}
		\lim\limits_{x\rightarrow k}\frac{\mathcal{P}_{n+1}^*(k;x)}{\mathcal{P}_n^*(k;x)}&=\frac{\mathcal{P}_n(k)}{\mathcal{P}_{n+1}(k)}\lim\limits_{x\rightarrow k}\left(\frac{\mathcal{P}_{n+2}(x)\mathcal{P}_{n+1}(k)-\mathcal{P}_{n+2}(k)\mathcal{P}_{n+1}(x)}{\mathcal{P}_{n+1}(x)\mathcal{P}_{n}(k)-\mathcal{P}_{n+1}(k)\mathcal{P}_{n}(x)}\right)\\
		&= \frac{\mathcal{P}_n(k)}{\mathcal{P}_{n+1}(k)}\left(\frac{\mathcal{P}_{n+2}'(k)\mathcal{P}_{n+1}(k)-\mathcal{P}_{n+2}(k)\mathcal{P}_{n+1}'(k)}{\mathcal{P}_{n+1}'(k)\mathcal{P}_{n}(k)-\mathcal{P}_{n+1}(k)\mathcal{P}_{n}'(k)}\right)\\
		&= \frac{\mathcal{P}_n(k)}{\mathcal{P}_{n+1}(k)}\lambda_{n+2}\left(\frac{\mathlarger{\sum}\limits_{j=0}^{n+1} \frac{\mathcal{P}_{j}^2(k)}{\lambda_1 \lambda_2...\lambda_{j+1}}}{\mathlarger{\sum}\limits_{j=0}^{n} \frac{\mathcal{P}_{j}^2(k)}{\lambda_1 \lambda_2...\lambda_{j+1}}}\right) \\
		&=\frac{\mathcal{P}_n(k)}{\mathcal{P}_{n+1}(k)}\lambda_{n+2}\left(1+\frac{\mathcal{P}_{n+1}^2(k)}{\lambda_1 \lambda_2...\lambda_{n+2}\mathlarger{\sum}\limits_{j=0}^{n} \frac{ \mathcal{P}_{j}^2(k)}{\lambda_1 \lambda_2...\lambda_{j+1}}}\right).
	\end{align*}
	In similar lines, we can obtain \eqref{kernel ratio 2}.	
	This completes the proof.
\end{proof}

\begin{corollary}\label{Ratio kernel Chebyshev}
	Let $\{\hat{\mathcal{C}}_n(x)\}_{n=0}^{\infty}$ be a sequence of monic Chebyshev polynomials of first kind. Then
	\begin{align*}
		\lim\limits_{n\rightarrow \infty}\lim\limits_{x\rightarrow 1}\frac{\hat{\mathcal{C}}_{n+1}^*(1;x)}{\hat{\mathcal{C}}_{n}^*(1;x)}=\frac{1}{2}.
	\end{align*}
\end{corollary}
\begin{proof}
	By using Theorem \ref{th:Kernel}, we have
	$$\lim\limits_{x\rightarrow 1}\frac{\hat{\mathcal{C}}_{n+1}^*(1;x)}{\hat{\mathcal{C}}_{n}^*(1;x)}=\frac{1}{2}\left(1+\frac{4}{2n+1}\right).$$
	Allowing $n\rightarrow\infty$, we get the desired result.
\end{proof}


Next, we will discuss the link between the ratio of kernel polynomials and infinite continued fractions. For this, we first need the definition of the hypergeometric functions.\\
The Gauss hypergeometric function 
${}_2F_1(p,q;r;z)$ is given by
\begin{align}
	F(p,q;r;z):= {}_2F_1(p,q;r;z)= \sum\limits_{n=0}^{\infty}\frac{(p)_n(q)_n}{(r)_nn!}z^n~ \text{ for } ~r\not\in \{0,-1,-2,...\},
\end{align}
where the symbol $(\cdot)_n$ is known as Pochhammer symbol  and is defined as
\begin{equation*}
	(p)_n=p(p+1)(p+2)...(p+n-1)=\frac{\Gamma(p+n)}{\Gamma p}, ~\text{with}~ (p)_0=1.
\end{equation*}

The above series converges absolutely in $\{z\in \mathbb{C}: |z| <1\}.$ Further we can analytically continue the series as a single valued function everywhere except any path joining the branch points 1 and infinity\cite{Andrews_Askey_Ranjan_Special functions}.

Note that if we take either $p$ or $q$ to be a negative integer, the terms of the series will vanish after some stage and we will be left with a finite linear combination of monomials. If this happens, the convergence of the hypergeometric series is not an issue.\\

If we replace $z$ by $z/q$ and allow $q \rightarrow \infty$, then by using
\begin{equation*}
	\lim\limits_{n\rightarrow \infty}\frac{(q)_n}{q^n}=1,
\end{equation*}  we obtain the \emph{Kummer or confluent hypergeometric function}
\begin{equation*}
	\phi(p;r;z):={}_1F_1(p;r; z) = \lim\limits_{q\rightarrow \infty}F(p, q; r; z/q) = \sum\limits_{n=0}^{\infty}\frac{(p)_n}{(r)_nn!}z^n~ \text{ for } ~r\not\in \{0,-1,-2,...\}.
\end{equation*}

We shall use the following contiguous relation satisfied by the Gaussian hypergeometric function to obtain the continued fraction of ratio of hypergeometric functions.
\begin{align}
	F(p+1,q;r;z)&=F(p,q;r;z)-\frac{q}{r}zF(p+1,q+1;r+1;z)\label{Contiguous R1},\\
	F(p,q;r;z)&=F(p,q+1;r+1;z)-\frac{p(r-q)}{r(r+1)}zF(p+1,q+1;r+2;z),\label{Contiguous R2}\\
	\nonumber F(p,q+1;r+1;z)&=F(p+1,q+1;r+2;z)\\
	&\hspace{2.5cm}-\frac{(q+1)(r-p+1)}{(r+1)(r+2)}zF(p+1,q+2;r+3;z)\label{Contiguous R3}.
\end{align}

We can use \eqref{Contiguous R1}-\eqref{Contiguous R3} to get the ratio of Gauss hypergeometric functions\cite[p. 337]{Wall}(see also \cite{Kustner 2002,Derevyagin_2018_JDEA}).
\begin{align}\label{Gauss}
	\frac{F(p+1, q;r; z)}{F(p, q; r; z)} = \frac{1}{1}\fminus\frac{\left(1-g_{0}\right) g_{1} z}{1}\fminus\frac{\left(1-g_{1}\right) g_{2} z}{1}\fminus\frac{\left(1-g_{2}\right) g_{3} z}{1}\fminus\cdots
\end{align}
with$$
g_{j}=g_{j}(p, q, r):= \begin{cases}0 & \text { for } j=0, \\ \frac{p+k}{r+2 k-1} & \text { for } j=2 k \geq 2, k \geq 1, \\ \frac{q+k-1}{r+2 k-2} & \text { for } j=2 k-1 \geq 1, k \geq 1.\end{cases}
$$
Hence, we can write the ratio of Kummer hypergeometric functions as a limit of the ratio of Gauss hypergeometric functions by

\begin{align}\label{ratio}
	\frac{\phi(p+1;r; z)}{\phi(p;r; z)}= \lim\limits_{q\rightarrow \infty}\frac{F(p+1, q;r; z/q)}{F(p, q; r; z/q)} = \frac{1}{1}\fminus\frac{d_1z}{1}\fminus\frac{d_2 z}{1}\fminus\frac{d_3z}{1}\fminus\cdots
\end{align}
with
$d_{j}= \lim\limits_{q\rightarrow \infty}\frac{\left(1-g_{j-1}\right)g_{j}}{q} $ for all $j \geq 1$. So
$$
d_{j}=d_{j}(p, r):= \begin{cases}\frac{1}{r} & \text { for } j=1, \\ \frac{-(p+k)}{(r+2 k-1)(r+2 k-2)} & \text { for } j=2 k,  k \geq 1, \\ \frac{r-p+k-1}{(r+2 k-1)(r+2 k-2)} & \text { for } j=2 k-1,  k \geq 2.\end{cases}
$$

If we put $p=-n, r= \gamma +2$ and $z=-x$ in \eqref{ratio}, we obtain
\begin{align}\label{Continued Laguerre}
	\frac{\phi(-n+1;\gamma +2; -x)}{\phi(-n;\gamma +2; -x)} = \frac{1}{1}\fplus\frac{\tilde{d}_1x}{1}\fplus\frac{\tilde{d}_2 x}{1}\fplus\frac{\tilde{d}_3x}{1}\fplus\cdots
\end{align} with

\begin{align}\label{CF coeffcients}
	\tilde{d}_{j}=\tilde{d}_{j}(-n, \gamma +2):= \begin{cases}\frac{1}{\gamma +2} & \text { for } j=1, \\ \frac{(n-k)}{(\gamma +2 k)(\gamma +2 k+1)} & \text { for } j=2 k, k \geq 1, \\ \frac{\gamma +n+k+1}{(\gamma +2 k)(\gamma +2 k+1)} & \text { for } j=2 k-1, k \geq 2.\end{cases}
\end{align}

\subsection{Kernel of Laguerre polynomials}
We know that Laguerre polynomials with parameter $\gamma$ can be written in the form of Kummer hypergeometric functions\cite{Bateman}.
$$
L_{n}^{\gamma}(x)=\left(\begin{array}{c}
	n+\gamma \\
	n
\end{array}\right){}_1F_{1}(-n ; \gamma+1 ; x), \quad n= 0,1,2,....
$$
$\{L_{n}^{\gamma}(x)\}_{n=0}^{\infty}$ forms an  orthogonal system  on $[0,+\infty)$ with respect to the weight function $w(x)=x^{\gamma} e^{-x}, \gamma>$ $-1.$

We can normalize the Laguerre polynomials and define 
$$
\mathbb{L}_{n}(x)=\mathbb{L}_{n}(\gamma ; x):=\frac{1}{\sqrt{\Gamma(\gamma+1)\left(\begin{array}{c}
			n+\gamma\\
			n
		\end{array}\right)}} L_{n}^{\gamma}(-x), \quad  n= 0,1,2,....
$$
$\{\mathbb{L}_{n}(x)\}_{n=0}^{\infty}$ forms an orthonormal  system on $(-\infty, 0]$ with respect to the weight function $(-x)^{\gamma} e^{x}$\cite{Szego}.
\newline
Considering \eqref{Other form of kernel polynomial} for the Laguerre polynomials with the particular value  $k=0$, we get 
$$L_n^{\gamma ^*}(0;x)= \lambda_1\lambda_2...\lambda_{n+1}(L_{n}^{\gamma}(0))^{-1}\mathbb{L}_{n}(\gamma+1 ; x),$$
where $\lambda_{n+1}=n(n+\gamma)$\cite[p. 154]{Chihara book}.
\newline
So, the ratio of the kernel of Laguerre polynomials for $k=0$ is given by

$$\frac{L_{n-1}^{\gamma ^*}(0;x)}{L_n^{\gamma ^*}(0;x)}=\frac{1}{\lambda_{n+1}}\frac{L_{n}^{\gamma}(0)\mathbb{L}_{n-1}(\gamma+1 ; x)}{L_{n-1}^{\gamma}(0)\mathbb{L}_{n}(\gamma+1 ; x)}.$$
We can write the above expression as
$$\frac{L_{n-1}^{\gamma ^*}(0;x)}{L_n^{\gamma ^*}(0;x)}=\frac{1}{n^2}\sqrt{\frac{\mathbb{B}(n,\gamma +2)}{n\mathbb{B}(n,\gamma +1)}}\frac{\phi(-n+1;\gamma +2; -x)}{\phi(-n;\gamma +2; -x)},$$
where $\mathbb{B}(\cdot ,\cdot)$ denotes the well-known Beta function.

Hence, we can use \eqref{Continued Laguerre} to obtain the ratio of kernel of Laguerre polynomials for $k=0$ in terms of the continued fractions as
$$\frac{L_{n-1}^{\gamma ^*}(0;x)}{L_n^{\gamma ^*}(0;x)}=\frac{1}{n^2}\sqrt{\frac{\mathbb{B}(n,\gamma +2)}{n\mathbb{B}(n,\gamma +1)}}\left(\frac{1}{1}\fplus\frac{\tilde{d}_1x}{1}\fplus\frac{\tilde{d}_2 x}{1}\fplus\frac{\tilde{d}_3x}{1}\fplus\cdots\right),$$
where $\tilde{d}_{j}$'s are given by \eqref{CF coeffcients}.

Similarly we can express the ratio of the kernels  of Laguerre polynomials with different parameteric value of $\gamma>0$ for $k=0$ as

$$\frac{L_{n-1}^{\gamma ^*}(0;x)}{L_n^{(\gamma-1) ^*}(0;x)}=\frac{\lambda_1\lambda_2...\lambda_{n}}{\tilde{\lambda}_1\tilde{\lambda}_2...\tilde{\lambda}_{n+1}}\frac{L_{n}^{(\gamma-1)}(0)\mathbb{L}_{n-1}(\gamma+1 ; x)}{L_{n-1}^{\gamma}(0)\mathbb{L}_{n}(\gamma ; x)},$$
where $\tilde{\lambda}_{n+1}=n(n+\gamma-1).$\\
The above ratio can be simplified as
\begin{align}
	\frac{L_{n-1}^{\gamma ^*}(0;x)}{L_n^{(\gamma-1) ^*}(0;x)}=\frac{\gamma^2}{n^{3/2}(n+\gamma)(\gamma +1)(n+\gamma -1)}\frac{\phi(-n+1;\gamma +2; -x)}{\phi(-n;\gamma +1; -x)}.
\end{align}
Now, we can use \cite[eq. 91.1]{Wall} to obtain
\begin{align}
	\frac{L_{n-1}^{\gamma ^*}(0;x)}{L_n^{(\gamma-1) ^*}(0;x)}=\frac{\gamma^2}{n^{3/2}(n+\gamma)(\gamma +1)(n+\gamma -1)}\left(\frac{1}{1}\fplus\frac{d'_1x}{1}\fminus\frac{d'_2 x}{1}\fplus\frac{d'_3x}{1}\fminus\cdots\right),
\end{align}
where $$d'_{j}=d'_{j}(n, r):= \begin{cases} \frac{n+k+\gamma +1}{(\gamma +2 k+1)(\gamma +2 k+2)} & \text { for } j=2 k+1, k \geq 0, \\ \frac{1-n+k}{(\gamma +2 k+1)(\gamma +2 k+2)} & \text { for } j=2 k+2 , k \geq 0.\end{cases}$$
\subsection{Kernel of Jacobi polynomials}
We know that the Jacobi polynomials with parameter $(\gamma,\delta)$ can be written in the form of Gauss hypergeometric functions \cite{Markett_2022_JAT}

$$
P_{n}^{(\gamma, \delta)}(x)=\left(\begin{array}{c}
	n+\gamma \\
	n
\end{array}\right)F\left(-n , n+\gamma+\delta+1;\gamma +1 ; \frac{1-x}{2}\right), \quad n \in \mathbb{Z}_{+}.
$$
Note that $\{P_{n}^{(\gamma,\delta)}(x)\}_{n=0}^{\infty}$ forms an  orthogonal system  on $[-1,1]$ with respect to the weight function $w(x)=(1-x)^{\gamma}(1+x)^{\delta}, \gamma,\delta >$ $-1.$\\
The normalisation 
\begin{align}
	\tilde{P}_n^{(\gamma ,\delta)}(x)= \sqrt{\frac{(2n+\gamma+\delta+1)\Gamma(n+1)\Gamma(n+\gamma+\delta+1)}{2^{\gamma +\delta +1}\Gamma(n+\gamma+1)\Gamma(n+\delta+1)}} P_n^{(\gamma, \delta)}(x)
\end{align}
provides the sequence $\{\tilde{P}_{n}^{(\gamma,\delta)}(x)\}_{n=0}^{\infty}$ as an  orthonormal system  on $[-1,1]$ with respect to the weight function $w(x)=(1-x)^{\gamma}(1+x)^{\delta}, \gamma,\delta >$ $-1.$
\newline
Considering \eqref{Other form of kernel polynomial} for the Jacobi polynomials with the particular value  $k=1$, we get 

$$P_n^{(\gamma,\delta) ^*}(1;x)= \lambda_1\lambda_2...\lambda_{n+1}(P_{n}^{(\gamma,\delta)}(1))^{-1}\tilde{P}_{n}^{(\gamma+1,\delta)} (x),$$
where \text{\cite[p. 153]{Chihara book}}
\begin{align*}
	\lambda_{n+1}=\frac{4n(n+\gamma)(n+\delta)(n+\gamma+\delta)}{(2n+\gamma+\delta)^2(2n+\gamma+\delta+1)(2n+\gamma +\delta-1)}.
\end{align*}
Hence, the ratio of the kernel of Jacobi polynomials with parameters $\gamma >-1, \delta>0$ for $k=1$ is given by
\begin{align*}
	\frac{P_{n-1}^{(\gamma+1,\delta) ^*}(1;x)}{P_n^{(\gamma+1,\delta -1) ^*}(1;x)}=\frac{\lambda_1\lambda_2...\lambda_{n}}{\tilde{\lambda}_1\tilde{\lambda}_2...\tilde{\lambda}_{n+1}}\frac{P_n^{(\gamma+1,\delta -1)}(1)\tilde{P}_{n-1}^{(\gamma +1,\delta)}(x)}{P_{n-1}^{(\gamma+1,\delta) }(1)\tilde{P}_{n}^{(\gamma +1,\delta-1)}(x)}.
\end{align*}
The above ratio can be simplified as
\begin{align}
	\frac{P_{n-1}^{(\gamma+1,\delta) ^*}(1;x)}{P_n^{(\gamma+1,\delta -1) ^*}(1;x)}=C(n,\gamma,\delta)\frac{F\left(-n+1 , n+\gamma+\delta+1;\gamma +2 ; \frac{1-x}{2}\right)}{F\left(-n , n+\gamma+\delta+1;\gamma +2 ; \frac{1-x}{2}\right)},
\end{align}
where

\begin{align*}
	C(n, \gamma ,\delta)=\sqrt{\frac{(\gamma +\delta +2)^2(2n+\gamma +\delta +1)(2n+\gamma +\delta )^3}{32n^3(n+\gamma +1)(\gamma +1)^2\delta^2}}.
\end{align*}
Hence, we can use \eqref{Gauss} to write the ratio of kernel of Jacobi polynomials in terms of continued fractions as
\begin{align}
	\nonumber \frac{P_{n-1}^{(\gamma+1,\delta) ^*}(1;x)}{P_n^{(\gamma+1,\delta -1) ^*}(1;x)}=C(n,\gamma,\delta)\left(\frac{1}{1}\fminus\frac{\left(1-e_{0}\right) e_{1} (\frac{1-x}{2})}{1}\fminus\frac{\left(1-e_{1}\right) e_{2} (\frac{1-x}{2})}{1}\right.\\
	\left. \fminus\frac{\left(1-e_{2}\right) e_{3} (\frac{1-x}{2})}{1}\fminus\cdots\right),
\end{align}
with
\begin{align}\label{CF-Jacobi-chain-coefficients}
e_{j}=e_{j}(n,\gamma,\delta):= \begin{cases}0 & \text { for } j=0, \\ \frac{-n+k}{\gamma+2 k+1} & \text { for } j=2 k, k \geq 1, \\ \frac{n+\gamma +\delta+k}{\gamma+2 k} & \text { for } j=2 k-1 , k \geq 1.\end{cases}
\end{align}
\section{Concluding remarks}
In this work the quasi-type kernel polynomials are introduced and one of our main objective that was established was the following. Given a quasi-type kernel polynomial, to find a suitable orthogonal polynomial which is an outcome of specific spectral transformation, whose linear combination with the quasi-type kernel polynomial recovers the orthogonality property. Besides this, several observations are made which are useful for future research and the same is outlined in this section.

The identity \eqref{OP in terms of Kernel} in Proposition $\ref{Prop:Quasi Kernel TTRR for CDKernel}$ was proved using TTRR \eqref{TTRR} and the same was established using Christoffel-Darboux kernel \eqref{CD identity} in \cite[eq. 2.5]{Kernel polynomials_Paco_2001}. Hence, it would be interesting to revisit many other results proved in the literature using Christoffel-Darboux kernel \eqref{CD identity}, and give an attempt to prove using TTRR.

	In the hypothesis of Theorem \ref{Recovery from iterated kernel and quasi-type kernel}, we required  the coefficients $\tilde{L}_n$ and $\tilde{M}_n$ of quasi-type kernel polynomial of order two to satisfy the expression \eqref{condition coeff. of quasi-type kernel order 2}. As a result, we obtained two unique sequences of constants $\alpha_n$ and $\beta_n$ that are useful in recovering the orthogonality given by the polynomial $\mathcal{P}_n(x)$. It would be interesting to remove the hypothesis of this particular choice of the coefficients $\tilde{L}_n$ and $\tilde{M}_n$. More specifically, we end this point of discussion with the following question.
	\begin{Problem}
		Is it possible to obtain three sequences of constants so that relaxation of the hypothesis \eqref{condition coeff. of quasi-type kernel order 2} is permissible?
		\end{Problem}

In theorems $\ref{QK+OP+KP}$, $\ref{QK+OP+Geronimus}$, $\ref{QK+OP+Uvarov}$ and $\ref{Recovery from iterated kernel and quasi-type kernel}$, the quasi-type Kernel polynomial is written in combination with a specific spectral transformed polynomial and it has been established that the resultant polynomial has the same orthogonality given by the moment functional ${\mathcal{L}}$. This leads to the question of decomposing the original orthogonal polynomial $\{P_n\}$, given by the moment functional ${\mathcal{L}}$ into the linear combination of quasi-type Kernel polynomial and another orthogonal polynomial, related to the given orthogonal polynomial and the relation between these decompositions. Hence we propose the following problem.
\begin{Problem}
To find the conditions under which an orthogonal polynomial can be decomposed into two parts, viz., a quasi-type kernel polynomial and a specific orthogonal polynomial related to the given polynomial.
\end{Problem}
It is expected that the decomposed part of the orthogonal polynomial, from the proved results, is a specific spectral transformation of the given orthogonal polynomial. However, it may be some other orthogonal polynomial with different properties, other than the spectral transformation of the given polynomial. Further, it is possible that the decomposed orthogonal polynomial and the quasi-type kernel polynomial have an orthogonality between them, leading to the biorthogonality property given in the sense of Konhauser\cite{Konhauser-biorthogonal-JMAA-1965}. For details of this biorthogonality, we refer to \cite{Kiran-Swami-PAMS, Konhauser-biorthogonal-JMAA-1965}. We formulate this as another problem.
\begin{Problem}
Given the decomposition of an orthogonal polynomial into its quasi-type kernel polynomial and another orthogonal polynomial, is there any biorthogonality relation between these two decomposed polynomials?
\end{Problem}

The $g_n$'s given by \eqref{Gauss} while finding the ratio of Gaussian hypergeometric functions constitute the $g$-sequence and hence the $g$-fraction, see \cite{Chihara book}. Hence, the $\tilde{d}_n$'s given by \eqref{CF coeffcients} for the ratio related to the Laguerre polynomials and the $e_n$'s given by \eqref{CF-Jacobi-chain-coefficients} for the ratio related to the Jacobi polynomials lead to the study of chain sequences \cite{Chihara book}. In fact, the sequence $\displaystyle\left\{\dfrac{\lambda_{n+1}}{c_nc_{n+1}}\right\}$ obtained from the TTRR \eqref{TTRR} is a chain sequence, for $c_n>0$, $n\geq 1$. A sequence $\{l_n\}$ that satisfies $l_n=(1-g_{n-1})g_n$, $n\geq 1$ is a positive chain sequence, where the $g_n$'s are called parameter sequence with $0\leq g_0<1$ and $0<g_n<1$ for $n\geq 1$ \cite{Chihara book}. Hence, given $c_n$, using this parameter sequence, we can find $\lambda_n$ and hence the TTRR \eqref{TTRR} can be formed and the sequence of orthogonal polynomials can be extracted for the given moment functional ${\mathcal{L}}$.  Further, the parameter sequence $\{g_n\}$ is called minimal parameter sequence and denoted by $\{m_n\}$, if $g_0: =m_0=0$. In fact, every parameter sequence has a minimal parameter sequence \cite[p.91-92]{Chihara book}. The sequence $\{M_n\}$ is called the maximal parameter sequence for the fixed chain sequence $\{l_n\}$, where
\begin{align*}
M_n = inf\{ g_n, \, \, {\mbox{for each}} \,  n, \{g_k\}\in {\mathcal{G}}\},
\end{align*}
with ${\mathcal{G}}$ to be the set of all parameter sequence $\{g_k\}$ of $\{l_n\}$.
If $m_n=M_n$, then the parameter sequence is unique and the chain sequence $\{l_n\}$ is called the Single Parameter Positive Chain Sequence or SPPCS in short.
For the details of this terminology, we refer to \cite{Chihara book, Costa_Ranga_OPUC and chain sequence_JAT_2013} and for recent results in this direction, we refer to \cite{2022_PAMS_Ranga_ChainSeq}.

Note that chain sequences are useful in studying various properties of the corresponding orthogonal polynomials including the moment problems. In this context, it would be useful, if it is a SPPCS. In case the chain sequence is not SPPCS, there are many ways of finding a SPPCS and one such method is given in \cite{2016-KiranRangaSwami-Symmetry-CCS}, where given a chain sequence $\{l_n\}$ the complementary chain sequence $\{k_n\}$ is defined as $k_n:=1-l_n$. It was established in \cite{2016-KiranRangaSwami-Symmetry-CCS} that either $\{l_n\}$ or $\{k_n\}$ must be a SPPCS. Hence we end this manuscript with the following problem which would provide an interesting future research in this direction.

\begin{Problem}
To find the nature of the SPPCS related to the sequences given by \eqref{CF coeffcients} and \eqref{CF-Jacobi-chain-coefficients} and its significance in studying the properties of the corresponding orthogonal polynomials.
\end{Problem}

\textbf{Acknowledgement}: The work of the second author is supported by the project No. CRG/2019/00200/MS of Science and Engineering Research Board, Department of Science and Technology, New Delhi, India.




\begin{thebibliography}{55}
\bibitem{Akhiezer 1965}N. I. Akhiezer, {\it The classical moment problem and some related questions in analysis}, translated by N. Kemmer, Hafner Publishing Co., New York, 1965.
\bibitem{When do linear}M. Alfaro, F. Marcellán, A. Pe$\tilde{\text{n}}$a and M. Luisa Rezola,  When do linear combinations of orthogonal polynomials yield new sequences of orthogonal polynomials?, J. Comput. Appl. Math. {\bf 233} (2010), no.~6, 1446--1452.
\bibitem{Alfaro2011 CMA}M. Alfaro, A. Pe$\tilde{\text{n}}$a,  M. Luisa Rezola and F. Marcellán,   Orthogonal polynomials associated with an inverse quadratic spectral transform, Comput. Math. Appl. {\bf 61} (2011), no.~4, 888--900.
\bibitem{Manuel Alfaro_1994}M. Alfaro and L. Moral, Quasi-orthogonality on the unit circle and semi-classical forms, Portugal. Math. {\bf 51} (1994), no.~1, 47--62.
 \bibitem{Andrews_Askey_Ranjan_Special functions}G. E. Andrews, R. Askey and R. Roy, {\it Special functions}, Encyclopedia of Mathematics and its Applications, 71, Cambridge University Press, Cambridge, 1999.
 \bibitem{Derevyagin_Bailey_CJM by DT} R. Bailey and M. Derevyagin,  Complex Jacobi matrices generated by Darboux transformations, arXiv preprint arXiv:2107.09824, (2021), 32 pages.
 \bibitem{Kiran-Swami-PAMS} {Kiran Kumar Behera and A. Swaminathan}, Biorthogonal rational functions of $R_{II}$ type,
{{Proc. Amer. Math. Soc.}},  {\bf 147} (7), (2019) 3061--3073.

 \bibitem{2016-KiranRangaSwami-Symmetry-CCS} K. K. Behera, A. Sri Ranga\ and\ A. Swaminathan, Orthogonal polynomials associated with complementary chain sequences, SIGMA Symmetry Integrability Geom. Methods Appl. {\bf 12} (2016), Paper No. 075, 17 pp.
 \bibitem{Bracciali_Ranga_Andrei_2018_JAT_CD Kernal on unit circle}C. F. Bracciali, A. Martin\'{e}z-Finkelshtein, A. Sri Ranga and D.O. Veronese,  Christoffel formula for kernel polynomials on the unit circle, J. Approx. Theory {\bf 235} (2018), 46--73.
 \bibitem{Castillo_2017_linear spectral_integral tranform}K. Castillo\ and\ M. N. Rebocho,  On linear spectral transformations and the Laguerre-Hahn class, Integral Transforms Spec. Funct. {\bf 28} (2017), no.~11, 859--875.
 \bibitem{Chihara}T. S. Chihara,  On quasi-orthogonal polynomials, Proc. Amer. Math. Soc. {\bf 8} (1957), 765--767.
  \bibitem{Chihara_1964_Kernel}T. S. Chihara,  On kernel polynomials and related systems, Boll. Un. Mat. Ital. (3) {\bf 19} (1964), 451--459.
 \bibitem{Chihara book}T. S. Chihara, {\it An introduction to orthogonal polynomials}, Mathematics and its Applications, Vol. 13, Gordon and Breach Science Publishers, New York, 1978.
 \bibitem{Costa_Ranga_OPUC and chain sequence_JAT_2013}M. S. Costa, H. M. Felix and A. Sri Ranga,  Orthogonal polynomials on the unit circle and chain sequences, J. Approx. Theory {\bf 173} (2013), 14--32.
 \bibitem{Derevyagin_2018_JDEA}M. Derevyagin,  Jacobi matrices generated by ratios of hypergeometric functions, J. Difference Equ. Appl. {\bf 24} (2018), no.~2, 267--276.
 \bibitem{Dickinson}D. Dickinson,  On quasi-orthogonal polynomials, Proc. Amer. Math. Soc. {\bf 12} (1961), 185--194.
 \bibitem{Draux1990}A. Draux,  On quasi-orthogonal polynomials, J. Approx. Theory {\bf 62} (1990), no.~1, 1--14.
 \bibitem{Draux2016} A. Draux,  On quasi-orthogonal polynomials of order $r$, Integral Transforms Spec. Funct. {\bf 27} (2016), no.~9, 747--765.
 \bibitem{Ostrovetski_1985_ Quasi OP Appr. sol ODE}V. K. Dzyadyk\ and\ L. A. Ostrovetski\u{\i},  Approximation by polynomials of boundary value problems for ordinary linear differential equations, Akad. Nauk Ukrain. SSR Inst. Mat. Preprint {\bf 1985}, no.~11, 31 pp
 \bibitem{Bateman} A. Erd\'{e}lyi, W. Magnus, F. Oberhettinger and F.G. Tricomi, {\it Higher transcendental functions. Vols. I, II}, McGraw-Hill Book Company, Inc., New York, 1953.
 \bibitem{Fejer}L. Fej\'{e}r,  Mechanische Quadraturen mit positiven Cotesschen Zahlen, Math. Z. {\bf 37} (1933), no.~1, 287--309.
 \bibitem{Gelfond_1954}A. Gelfond,  On polynomials deviating least from zero along with their derivatives, Doklady Akad. Nauk SSSR (N.S.) {\bf 96} (1954), 689--691.
 \bibitem{Paco_2018_conference}J. C. Garc\'ia-Ardila, F. Marcellán and M. Marriaga, From standard orthogonal polynomials to Sobolev orthogonal polynomials: The role of
 	semiclassical linear functionals. In Orthogonal Polynomials. 2nd AIMS-Volkswagen Stiftung Workshop, Douala, Cameroon, 5-12 October,
 2018, M. Foupouagnigni, W. Koepf Editors. Series Tutorials, Schools, and Workshops in the Mathematical Sciences (TSWMS),
 Birkhauser, Cham. 245-292 (2020).
 \bibitem{Garcia_Paco_2022_JMAA}J. C. Garc\'{\i}a-Ardila, F. Marcell\'{a}n\ and\ P. H. Villamil-Hern\'{a}ndez,  Associated orthogonal polynomials of the first kind and Darboux transformations, J. Math. Anal. Appl. {\bf 508} (2022), no.~2, 26 pp, https://doi.org/10.1016/j.jmaa.2021.125883.
 \bibitem{Grinshpun2004}Z. Grinshpun,  Special linear combinations of orthogonal polynomials, J. Math. Anal. Appl. {\bf 299} (2004), no.~1, 1--18.
 \bibitem{2022_PAMS_Ranga_ChainSeq} G. A. Marcato, A. Sri Ranga\ and\ Y. C. Lun, Parameters of a positive chain sequence associated with orthogonal polynomials, Proc. Amer. Math. Soc. {\bf 150} (2022), no.~6, 2553--2567.
 \bibitem{Positive definite of Uvarov}M. Humet and M. Van Barel,  When is the Uvarov transformation positive definite?, Numer. Algorithms {\bf 59} (2012), no.~1, 51--62.
 \bibitem{Ismail_2019_quasi-orthogonal}M. E. H. Ismail\ and\ X.-S. Wang,  On quasi-orthogonal polynomials: their differential equations, discriminants and electrostatics, J. Math. Anal. Appl. {\bf 474} (2019), no.~2, 1178--1197.
 \bibitem{Konhauser-biorthogonal-JMAA-1965}
J. D. E. Konhauser, Some properties of biothogonal polynomials, J. Math. Anal. Appl., {\bf 11} (1965), 242--260.
\bibitem{Kustner 2002}	R. K\"{u}stner,  Mapping properties of hypergeometric functions and convolutions of starlike or convex functions of order $\alpha$, Comput. Methods Funct. Theory {\bf 2} (2002), 597--610.
 \bibitem{Kernel polynomials_Paco_2001}K. H. Kwon, D.W. Lee, F. Marcellán and S.B. Park,  On kernel polynomials and self-perturbation of orthogonal polynomials, Ann. Mat. Pura Appl. (4) {\bf 180} (2001), no.~2, 127--146.
 \bibitem{M.M} F. Marcell\'{a}n and P. Maroni,  Sur l'adjonction d'une masse de Dirac \`a une forme r\'{e}guli\`ere et semi-classique, Ann. Mat. Pura Appl. (4) {\bf 162} (1992), 1--22.
 \bibitem{Markett_2022_JAT} C. Markett, The differential equation for Jacobi-Sobolev orthogonal polynomials with two linear perturbutions, J. Approx. Theory {\bf 280} (2022), 24 pp, https://doi.org/10.1016/j.jat.2022.105782.
 \bibitem{Natanson_CFT_vol2}I. P. Natanson, {\it Constructive function theory. Vol. II}, translated from the Russian by John R. Schulenberger, Frederick Ungar Publishing Co., New York, 1965.
\bibitem{Riesz} M. Riesz,  Sur le probleme des moments, Troisième note, Ark. Mat. Astr. Fys., {\bf17} (1923), 1-52.
\bibitem{Swiderski_Trojan_Asymptotic CD Kernel via TTRR_2021_JAT}G. \'{S}widerski\ and\ B. Trojan, Asymptotic behavior of Christoffel-Darboux kernel via three-term recurrence relation II, J. Approx. Theory {\bf 261} (2021), 48 pp., https://doi.org/10.1016/j.jat.2020.105496.
\bibitem{Swiderski_Assche_Christoffel_MOP_2022}G. \'{S}widerski\ and\ W. Van Assche, Christoffel functions for multiple orthogonal polynomials, J. Approx. Theory {\bf 283} (2022), 22 pp., https://doi.org/10.1016/j.jat.2022.105820.
\bibitem{Szego}G. Szeg\"{o}, {\it Orthogonal polynomials}, American Mathematical Society Colloquium Publications, Vol. 23, American Mathematical Society, Providence, RI, 1959.
\bibitem{Shohat} J. Shohat,  On mechanical quadratures, in particular, with positive coefficients, Trans. Amer. Math. Soc. {\bf 42} (1937), no.~3, 461--496.
\bibitem{Uvarov trans. first paper}V. B. Uvarov,  The connection between systems of polynomials that are orthogonal with respect to different distribution functions, \v{Z}. Vy\v{c}isl. Mat i Mat. Fiz. {\bf 9} (1969), 1253--1262.
\bibitem{Wall} H. S. Wall, {\it Analytic Theory of Continued Fractions}, D. Van Nostrand Company, Inc., New York, NY, 1948.


\end{thebibliography}
\end{document}